\documentclass[onefignum,onetabnum]{siamonline171218}
\usepackage{amsmath,amsfonts,amssymb,bm}
\usepackage{mathtools}
\usepackage{algorithm}
\usepackage{algpseudocode}
\usepackage{enumitem}
\usepackage{tikz}
\usetikzlibrary{arrows.meta, positioning, decorations.pathreplacing}

\newtheorem{remark}{Remark}
\newtheorem{assumption}{Assumption}

\newcommand{\fXYA}{f_{\boldsymbol{xy}}}
\newcommand{\fXYcondA}{f_{\boldsymbol{xy} \mid \boldsymbol{\alpha}}}

\newcommand{\fXYcondAwithArg}{f_{\boldsymbol{xy} \mid \boldsymbol{\alpha}}(\boldsymbol{X}, \boldsymbol{Y}, t \mid \boldsymbol{\alpha})}
\newcommand{\fXcondAwithArg}{f_{\boldsymbol{x} \mid \boldsymbol{\alpha}}(\boldsymbol{X}, t \mid \boldsymbol{\alpha})}

\newcommand{\RHSwithArg}{\boldsymbol{R}(\boldsymbol{q}(t), t;\boldsymbol{\alpha})}

\newcommand{\qw}[1]{{\color{black}{#1}}}
\newcommand{\del}[1]{}
\newcommand{\dpEul}[1]{\left(\partial_{\boldsymbol{\alpha}}#1\right)_{\boldsymbol{Q}_T}}
\newcommand{\dpLag}[1]{\left(\partial_{\boldsymbol{\alpha}}#1\right)_{\boldsymbol{q}_0}}

\title{Characteristic Sensitivity Ensembles for Inference of Hidden Dynamics from Marginal Observations}

\author{
Qi Wang\thanks{Department of Aerospace Engineering, San Diego State University, San Diego, CA, USA (\texttt{qwang4@sdsu.edu})}
\and
Gustaaf Jacobs\thanks{Department of Aerospace Engineering, San Diego State University, San Diego, CA, USA (\texttt{gjacobs@sdsu.edu})}
}

\begin{document}
\maketitle

\begin{abstract}
A framework is developed for the inference of dynamics described by a  generalized system of ordinary differential equations. A stochastic gradient method is coined that  infers dynamics from observed marginal probability density functions using the joint probability density function of the observable and latent variables.
Diffusion and other irreversible processes observed in a low-dimensional state can be recast as deterministic, reversible flows in a sufficiently augmented state space, where the joint density satisfies the hyperbolic Liouville equation. The marginal distribution observed is the projection of these hyperbolic dynamics onto the observed coordinates, with the latent components carrying the randomness and memory. This reframing allows inference for irreversible or stochastic dynamics into the recovery of a deterministic Ordinary Differential Equation (ODE) from marginal observations.
Instead of solving the high-dimensional Liouville equation for the joint density, the algorithm exploits its characteristic representation.
Particles sampled from the initial distribution are transported along characteristic lines.
The Eulerian sensitivity with respect to parameters is obtained by sensitivity propagation along the characteristic lines, with a crossed U-statistic producing an unbiased gradient estimator, which enables stochastic gradient descent.
Four experiments validate the method: recovery of a three-mode linear system observed through the marginal of a single mode; a nonlinear Gompertz growth model with a hidden mode; a bistable system whose hidden mode turns a unimodal marginal bimodal; and Stokes--Oseen drag law recovery for particles in a cellular flow.
Convergence behavior is analyzed across these settings.
\end{abstract}



\maketitle

\section{Introduction}
\label{sec:intro}

In machine learning methods and data-driven models, it is common to parametrically encode an unknown independent, potentially random variable in a response surface, flow map, and/or dynamical system with the purpose of inferring or learning it from observables.  Defining the observable variable as $\boldsymbol{x}(t) \in \mathbb{R}^{d_x}$, a generalized dynamical system of linear or non-linear ODEs  that encodes an unknown parametric vector  $\boldsymbol{\alpha} \in \mathbb{R}^p$   can be defined as  
\begin{eqnarray}
\dot{\boldsymbol{q}}(t) &=& \RHSwithArg, \\
\boldsymbol{q}(0) & \coloneqq &  \boldsymbol{q}_0 
\label{eqn:ODE}
\end{eqnarray}
where $\boldsymbol{q}(t) = (\boldsymbol{x}(t), \boldsymbol{y}(t)) \in \mathbb{R}^{d}$ is the state vector and $ \boldsymbol{q}_0$ its initial condition. The state vector may depend on the latent variable $\boldsymbol{y}(t) \in \mathbb{R}^{d_y}$, which describes the variables that are not observable or hidden. The dimension of the state vector $d=d_x+d_y$. The right-hand side $\boldsymbol{R}:\mathbb{R}^d \times \mathbb{R}^+ \to \mathbb{R}^d$ is deterministic and depends on the state variable conditioned on the deterministic $\boldsymbol{\alpha}$.

If the initial condition is known only statistically; for example, for chaotic velocity fields, clouds of tracers or for initial conditions that are known with a confidence interval, then  $\boldsymbol{q}_0$ and thus $\boldsymbol{q}(t)$ is described by a probability density function, $\fXYcondAwithArg $. 
Here the lower-case symbols $\boldsymbol{x}(t),\boldsymbol{y}(t)$ and $\boldsymbol{q}(t)$ denote deterministic trajectories of state variables, while upper-case symbols $\boldsymbol{X},\boldsymbol{Y}$ and $\boldsymbol{Q}$ denote the corresponding state-space coordinates used to describe the ensemble distribution of such trajectories.
Using the Liouville theorem and \eqref{eqn:ODE} 
a corresponding Liouville equation 
\begin{eqnarray}
\frac{\partial \fXYcondA}{\partial t} + \nabla_{\boldsymbol{Q}} \cdot \bigl(\fXYcondA \boldsymbol{R}\bigr) = 0, \\
\fXYcondA(\boldsymbol{Q}, 0)\coloneqq f_0,
\label{eq:Liouville_general}
\end{eqnarray}
can be derived  (e.g. \cite{pavliotis2014stochastic, 10.1063/5.0207403}), which describes the transport of  $\fXYcondAwithArg $.
The same vector field \(\boldsymbol{R}\) is evaluated either along a deterministic trajectory,
as \(\boldsymbol{R}(\boldsymbol{q}(t),t;\boldsymbol{\alpha})\), or at a state-space
coordinate, as \(\boldsymbol{R}(\boldsymbol{Q},t;\boldsymbol{\alpha})\), as used in the Liouville equation.
The operator $\nabla_{\boldsymbol{Q}}$ denotes differentiation in the state space\del{variable}. 
Note that the observable probability density function, $f_{\boldsymbol{x}|\boldsymbol{\alpha}}(\boldsymbol{X}, t|\boldsymbol{\alpha})$ is   obtained as the marginal of the joint pdf as follows:
\begin{equation} f_{\boldsymbol{x}|\boldsymbol{\alpha}}(\boldsymbol{X}, t|\boldsymbol{\alpha})
= \int_{\mathbb{R}^{d_y}} \fXYcondAwithArg\, d\boldsymbol{Y}. 
\end{equation}

\begin{figure}[h!]
    \centering
    \includegraphics[width=0.9\linewidth]{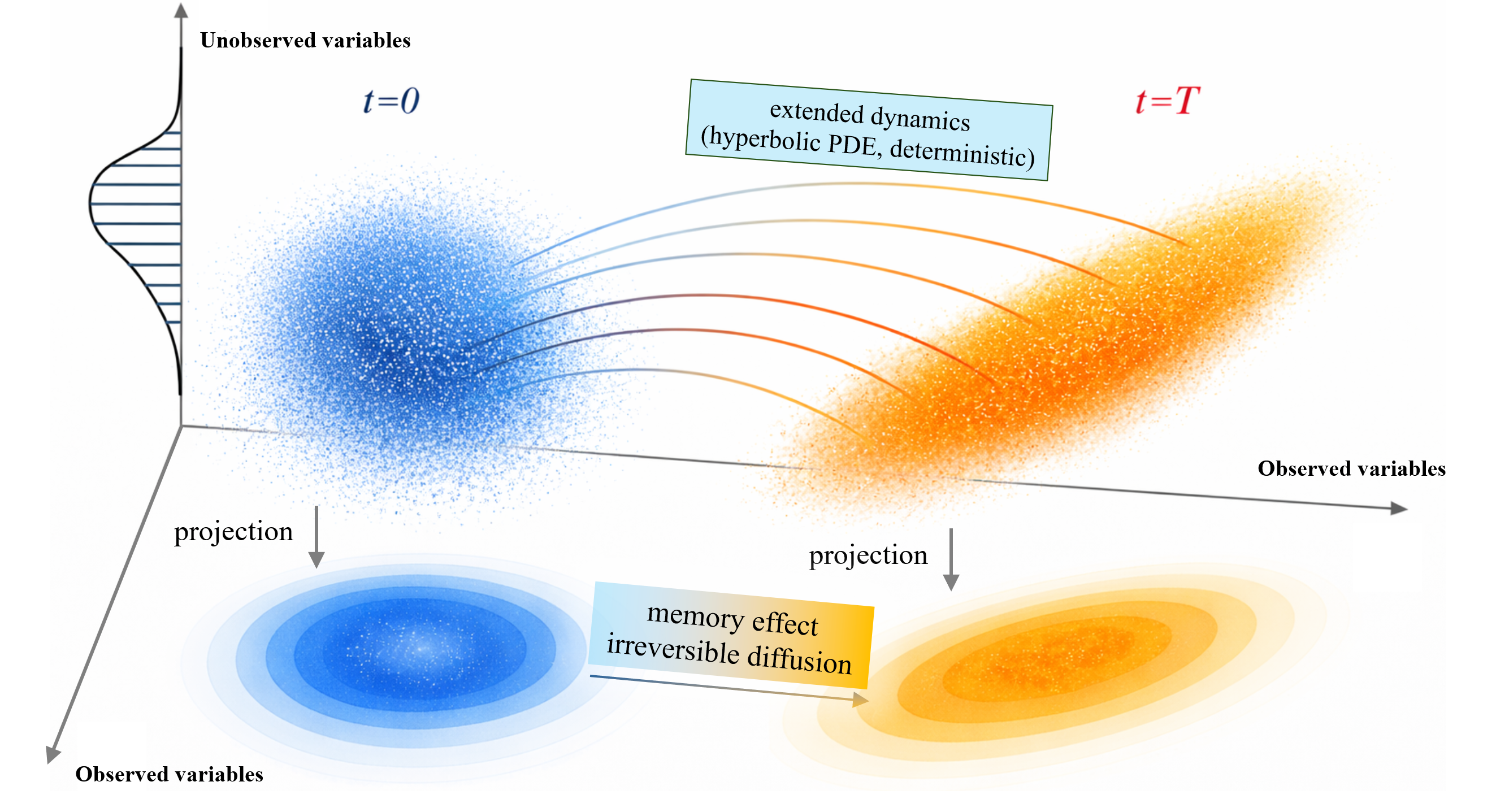}
    \caption{Deterministic augmented dynamics in extended phase space versus the irreversible, memory-dependent projected dynamics in the observed space.}
    \label{fig:schematic}
\end{figure}

In data-driven methods the goal is to infer $\boldsymbol{\alpha}$ from the observable density, $f_{\boldsymbol{x}|\boldsymbol{\alpha}}(\boldsymbol{X}, t|\boldsymbol{\alpha})$, as schematically illustrated in figure \ref{fig:schematic}. The role of the latent joint probability density function and the marginalization to the observable probability density function in methods of inference has been investigated in the context of a wide range of fields and applications. In operator theory, for example, Sz.-Nagy's dilation theorem \cite{nagy1970harmonic} guarantees that every contraction (equivalent to marginalization) on a Hilbert space is the compression of a unitary operator acting on an enlarged space (or the higher dimensional space of the joint density). In quantum mechanics, every open (dissipative) system can be modeled as a closed (unitary) system coupled to an environment \cite{breuer2002open}.
The Mori-Zwanzig formalism \cite{mori1965transport,zwanzig1961memory} provides a concrete expression of this principle in dynamical systems: if the hidden variables $\boldsymbol{y}$ are projected out, the observable $\boldsymbol{x}$ satisfies a generalized Langevin equation with memory, colored noise, and effective stochasticity absent in the original full dynamics.
The same structure appears in probabilistic modeling: hidden Markov models \cite{rabiner1989hmm}, latent ODEs \cite{rubanova2019latent}, and latent diffusion models \cite{rombach2022high} all introduce latent variables to explain non-Markovian or non-Gaussian behavior in observed processes.

While the lifted or marginal density of non-linear dynamics are governed by a closed forward system, the inverse solution is substantially more challenging, i.e, the inference of the parameters $\boldsymbol{\alpha}$ of the full-dimensional deterministic system, given only partial observations of the marginal density at discrete times, is inherently ill-posed. A  projection onto the observed variables loses information for the hidden degrees of freedom, which means a unique reconstruction is not guaranteed. In practice, the marginal density is observed only through empirical distribution or summary statistics such as moments. Furthermore, the projected dynamics obey a non-Markovian integro-differential equation, limiting the applicability of classical control and inverse methods.

Many methods for parameter inference for dynamical systems are formulated based on a state or parameter estimation driven by trajectory-level measurements, using adjoint methods \cite{kgb6-k3zm, Wang_Zaki_2025, lions1971optimal,dominguez2022inference}, ensemble Kalman filters \cite{evensen2009data}, or variational data assimilation \cite{Wang_Wang_Zaki_2022, kushner2003stochastic}.
These approaches typically rely on measurements that can be associated with trajectories. While distributions can be sampled using an ensemble, these methods do not directly accommodate marginal observations without trajectory correspondence.

A natural alternative is to work with a probability density function, using the Liouville equation for deterministic dynamics or the Fokker-Planck equation for stochastic dynamics \cite{risken1996fokkerplanck,10.1063/5.0207403}.
Direct inference that matches predicted and observed densities is computationally infeasible in high dimensions, because the density lives in the full lifted phase space $\mathbb{R}^d$; for most real-world applications, grid-based density solvers are out of the question.
Moment-closure and reduced models project the density evolution onto low-dimensional summaries \cite{sun2019moment}, but they introduce modeling bias that is difficult to control systematically for nonlinear dynamics.

Data-driven methods face analogous difficulties.
Neural ODEs \cite{chen2018neural} and latent ODEs \cite{rubanova2019latent}
learn dynamical systems by fitting trajectories or likelihoods of observed states. Their training relies on auxiliary inference networks to reconstruct the latent dynamics, and their training typically optimizes a variational lower bound (ELBO) involving an additional likelihood model.
Marginal distribution observations fall out of the scope and generally require additional modeling assumptions.
Continuous normalizing flows \cite{papamakarios2021normalizing,kobyzev2020normalizing} evolve densities along ODE trajectories effectively by solving the density factor along the trajectories, but they learn transport maps between fully observed distributions rather than inferring dynamics from marginal projections.
Simulation-based inference methods \cite{cranmer2020frontier} circumvent the density PDE by comparing simulated and observed summary statistics, but they typically require many forward simulations and provide no gradient information.
Stein variational gradient descent \cite{liu2016stein} uses kernel-based density matching to transport particles toward a target distribution, but assumes the target score is known.

The Schr\"odinger bridge problem (SBP) provides a complementary framework based on distributions.
The classical SBP can be viewed as a stochastic version of the optimal mass transport (OMT) problem \cite{benamou2000computational}, and seeks a stochastic process that matches two prescribed marginal distributions at fixed times while remaining as close as possible to a reference process, typically Brownian motion. 
Recent neural variants \cite{nodozi2024neural} parameterize the drift and diffusion with neural networks and train them with a Sinkhorn-style loss function.
The resulting Euler-Lagrange conditions of the Schr\"odinger bridge problem consist of a forward Fokker-Planck equation coupled with a backward equation governing the log-density gradient.
These coupled equations are solved numerically using physics-informed neural networks (PINNs) \cite{farea2024understanding}.
Closely related ideas appear in score-based diffusion models and flow-matching methods, which also construct stochastic or probability flow dynamics consistent with prescribed distributions. In particular, the score-based diffusion model \cite{song2020sde,song2021score} successfully exploits the irreversible structure within an irreversible process. Anderson's reversal theorem \cite{anderson1982reverse} guarantees that the spatial score $\nabla_{\boldsymbol{x}}\log p_t(\boldsymbol{x})$ alone is enough to invert an apparently irreversible noising process.
Learning this single log-density derivative is what turns the forward noising into a tractable backward generative model.
Flow-matching methods \cite{lipman2023flow} achieve the same forward objective by sidestepping log-densities and regressing the velocity field on interpolant pairs.
These approaches rely on an underlying stochastic reference dynamics and are not designed to infer latent dynamics from partial observations or to handle projection-induced non-Markovian effects.


In \cite{daniel2024liouville}, it was shown that a randomly forced system of ordinary differential equations gives rise to a hyperbolic Liouville equation, which can be used to model systems with non-Gaussian uncertainties and diffusion processes. More generally, many dynamical particle systems with random initial conditions are governed by hyperbolic equations for the probability density function (PDF). For example, the Vlasov equation describes the evolution of the PDF of collisionless plasmas \cite{birdsall}. The characteristics of these hyperbolic equations transport the PDF along characteristic trajectories and admit a well-posed inverse mapping. Indeed, tracing solutions along characteristics in both forward and backward time forms the basis of characteristic-based numerical methods, including the forward characteristic method described in \cite{dominguez2024high} and semi-Lagrangian methods \cite{natarajan2021high}. 
The characteristic formulation also offers a potential way to mitigate the curse of dimensionality, since only individual trajectories need to be evolved rather than discretizing the full phase space. Moreover, unlike moment-based approaches, it avoids the closure problem because the complete PDF is represented through its characteristic flow.

In this paper, a method for inferring latent dynamics is proposed based on a hyperbolic lifting of the marginal probability density function of observables. The method
leverages the characteristic form of the Liouville equation governing the joint probability density function. 
Characteristics, sampled from a random initial distribution, are used to determine the Eulerian derivative with respect to the parameter $\boldsymbol{\alpha}$
by freezing the state space solution  $\boldsymbol{q}_T$ at the observation time, $T$, and performing an automatic differentiation on the backward characteristic lines.
A weak loss that uses test functions on the predicted and observed marginal densities, and an exactly unbiased stochastic gradient is constructed via a crossed U-statistic over small sets of $N$ characteristics. A theorem is developed that provides an estimate of the sensitivity of the set with respect to the parameter along the backward characteristic.
The gradient variance scales as $O(1/N)$ and grows with the time horizon at a rate governed by the Lipschitz constant of the vector field, a phenomenon related to the growth of tangent dynamics and Lyapunov exponents in chaotic systems \cite{leith1974theoretical}.
Full-matrix preconditioning handles the anisotropic gradient noise that arises from partially observed dynamics. We prove convergence of the resulting stochastic gradient iteration.
For identifiable parametric models, the convergence rate is governed by the condition number of the sensitivity Jacobian, scaling quadratically with its inverse conditioning.

The paper is organized as follows. In section ~\ref{sec:forward_characteristic}, the characteristic form of the probability density function equation and its solution conditioned on the parameter $\boldsymbol{\alpha}$ are discussed. The inverse problem is posed in section~\ref{sec:problem_setup} followed by the development of the theorem for the characteristic-based sensitivity formula, a weak loss function, and a local identifiability analysis.  After a summary of the algorithm, Section~\ref{sec:algorithm} proves convergence and discusses stability and accuracy. Section~\ref{sec:experiments} presents the four numerical experiments.

\section{Characteristic forward problem}
\label{sec:forward_characteristic}

\paragraph{Notation}
In the Liouville formulation, uppercase symbols $\boldsymbol{Q}=(\boldsymbol{X},\boldsymbol{Y})$ denote Eulerian  arguments or coordinates of the probability density function, whereas lowercase symbols $\boldsymbol{q}(t)=(\boldsymbol{x}(t),\boldsymbol{y}(t))$ denote Lagrangian trajectories or sampled particles.
At the terminal time $T$, $\boldsymbol{Q}_T$ can thus be interpreted as  a prescribed terminal coordinate held fixed under an Eulerian parameter derivative, while $\boldsymbol{q}_0$ and $\boldsymbol{q}_T$ are sampled initial and terminal states.

Denoting differentiation in the parameter $\boldsymbol{\alpha}$ as $\partial_{\boldsymbol{\alpha}}$, we assume local regularity as follows:
\begin{assumption}[Local regularity]
\label{ass:char_regularity}
The admissible parameter set $\mathcal{A}\subset\mathbb{R}^p$ is open.
The vector field $\boldsymbol{R}:\mathbb{R}^d\times[0,T] \times \mathcal{A}\to\mathbb{R}^d$ is $C^1$ in $(\boldsymbol{q},t,\boldsymbol{\alpha})$ and globally Lipschitz in $\boldsymbol{q}$ uniformly in $(t,\boldsymbol{\alpha})$ with Lipschitz constant $L_R:=\sup_{(\boldsymbol{q},t,\boldsymbol{\alpha})}\|\nabla_{\boldsymbol{q}}\boldsymbol{R}(\boldsymbol{q},t;\boldsymbol{\alpha})\|<\infty$.
The divergence $h_{\boldsymbol{\alpha}}(\boldsymbol{q},t):=\nabla_{\boldsymbol{q}}\!\cdot\boldsymbol{R}(\boldsymbol{q},t;\boldsymbol{\alpha})$ is $C^1$ in $(\boldsymbol{q},\boldsymbol{\alpha})$, with $\partial_{\boldsymbol{\alpha}}h_{\boldsymbol{\alpha}}$ and $\nabla_{\boldsymbol{q}}h_{\boldsymbol{\alpha}}$ uniformly bounded in $(t,\boldsymbol{\alpha})$ on every bounded set of states.
The probability density  function integrates to unity, $\int_{\mathbb{R}^d}\fXYA(\boldsymbol{Q},0)\,d\boldsymbol{Q}=1$,  is $C^1$ and strictly positive on the interior $\Omega^{\!\circ}$ of its support $\Omega\subseteq\mathbb{R}^d$, with $\mathbb{E}_{\boldsymbol{q}_0\sim f_0}\|\nabla\log f_0(\boldsymbol{q}_0)\|^2<\infty$, the gradient being taken on $\Omega^{\!\circ}$ (which has full $f_0$-measure).
\end{assumption}

Under Assumption~\ref{ass:char_regularity}, the Picard-Lindel\"of theorem~\cite{murray2013existence} yields, for every $(\boldsymbol{q}_0,\boldsymbol{\alpha})\in\mathbb{R}^d\times\mathcal{A}$, a unique global solution of~\eqref{eqn:ODE} described by the  flow map
\begin{equation}
\Phi_t(\cdot;\boldsymbol{\alpha}):\mathbb{R}^d\to\mathbb{R}^d,
\qquad
\Phi_t(\boldsymbol{q}_0;\boldsymbol{\alpha}):=\boldsymbol{q}(t;\boldsymbol{q}_0,\boldsymbol{\alpha}),
\label{eq:flow_map}
\end{equation}
that is  continuously differentiable and depends  on $(\boldsymbol{q}_0,\boldsymbol{\alpha})$ and is, for each fixed $(\boldsymbol{\alpha},t)\in\mathcal{A}\times[0,T]$, a $C^1$ diffeomorphism with $C^1$ inverse $\Phi_t^{-1}(\cdot;\boldsymbol{\alpha})$.
The probability density function $\fXYcondA((\cdot,t)\mid\boldsymbol{\alpha}):=(\Phi_t(\cdot;\boldsymbol{\alpha}))_{\#}f_0$ in $C^1(\mathbb{R}^d\times[0,T])$ is the solution of the Liouville equation (\ref{eq:Liouville_general}).
The equivalent non-conservative form of the Liouville equation, 
is obtained using the product rule leading to
\begin{equation}
\bigl(\partial_t+\boldsymbol{R}\!\cdot\!\nabla_{\boldsymbol{Q}}\bigr)\,\fXYcondA \coloneqq \frac{D \fXYcondA }{Dt}
=
-h_{\boldsymbol{\alpha}}\,\fXYcondA,
\label{eq:liouville-nonconservative}
\end{equation}
The operator $D/Dt$ is the Lagrangian material derivative that transports $\fXYcondA$ along characteristic manifolds in the $(\boldsymbol{q}, t)$ space with characteristic velocity $\boldsymbol{R}$.
The source term $-h_{\boldsymbol{\alpha}}\fXYA$ on the right-hand side is interpreted as the effect of the local volume change on the flow map.

Note that the solution to the dynamical system~\eqref{eqn:ODE}  coincides with the characteristic manifold  of the
Liouville equation.
Therefore, the flow map
$\Phi_t(\boldsymbol{q}_0;\boldsymbol{\alpha})$ is a characteristic trace through the Eulerian coordinate
$\boldsymbol{Q}$, i.e., $\boldsymbol{Q}=\Phi_t(\boldsymbol{q}_0;\boldsymbol{\alpha})=\boldsymbol{q}(t)$, along
every characteristic. 
Thus, $\boldsymbol{Q}$ and $\boldsymbol{q}(t)$ have the same value along a characteristic but they should not be used interchangeably as the variables have different interpretations: $\boldsymbol{Q}$ marks the independent coordinate on which the probability density function depends, whereas $\boldsymbol{q}(t)$ marks the sampled trajectory along which it is transported.

Using this property of the characteristic form of the probability density function equation,  we can find the following characteristic solution:
\begin{lemma}[Characteristic representation of the Liouville solution]
\label{lem:char_representation}
Under Assumption~\ref{ass:char_regularity}, the probability  density function conditioned on the initial parameter $\alpha$  can be determined from its initial condition, for every $\boldsymbol{q}_0\in\Omega^{\!\circ}$, $T\geq 0$, and $\boldsymbol{\alpha}\in\mathcal{A}$, as follows,
\begin{equation}
\label{eq:char-solution}
\fXYcondA\bigl(\Phi_T(\boldsymbol{q}_0;\boldsymbol{\alpha}),T\mid\boldsymbol{\alpha}\bigr)
=
f_0(\boldsymbol{q}_0)\,
\exp\!\Bigl(-\!\int_0^T h_{\boldsymbol{\alpha}}\bigl(\Phi_s(\boldsymbol{q}_0;\boldsymbol{\alpha}),s\bigr)\,ds\Bigr).
\end{equation}
Equivalently, for every $\boldsymbol{q}(T)\coloneqq\boldsymbol{q}_T\in\mathbb{R}^d$ and $\boldsymbol{q}_0=\Phi_T^{-1}(\boldsymbol{q}_T;\boldsymbol{\alpha})$,
\begin{equation}
\label{eq:char-solution-terminal}
\log \fXYcondA(\boldsymbol{Q}_T,T \mid \boldsymbol{\alpha})
=
\log f_0(\boldsymbol{q}_0)
-\int_0^T h_{\boldsymbol{\alpha}}\bigl(\Phi_s(\boldsymbol{q}_0;\boldsymbol{\alpha}),s\bigr)\,ds.
\end{equation}
\end{lemma}

\begin{proof}
Set $\eta(t):=\fXYcondA(\Phi_t(\boldsymbol{q}_0;\boldsymbol{\alpha}),t\mid\boldsymbol{\alpha})$, so that $\eta(0)=f_0(\boldsymbol{q}_0)>0$.
Along the characteristic,~\eqref{eq:liouville-nonconservative} reads $\dot\eta(t)=-h_{\boldsymbol{\alpha}}(\Phi_t,t)\,\eta(t)$, where $\eta$ stays strictly positive on $[0,T]$ and $\log\eta$ is $C^1$ with $\frac{d}{dt}\log\eta=-h_{\boldsymbol{\alpha}}(\Phi_t,t)$.
Integration over $[0,T]$ yields~\eqref{eq:char-solution}; substituting $\boldsymbol{q}_0=\Phi_T^{-1}(\boldsymbol{q}_T;\boldsymbol{\alpha})$ obtains~\eqref{eq:char-solution-terminal}.
\end{proof}

\section{Inverse problem - parameter inference}
\label{sec:problem_setup}

The lifted state splits into an observed and a latent block, written $\boldsymbol{q}=(\boldsymbol{x},\boldsymbol{y})$ along a trajectory and $\boldsymbol{Q}=(\boldsymbol{X},\boldsymbol{Y})$ in state space, with $\boldsymbol{x},\boldsymbol{X}\in\mathbb{R}^{d_x}$ observed, $\boldsymbol{y},\boldsymbol{Y}\in\mathbb{R}^{d_y}$ latent, and $d=d_x+d_y$.
The canonical projection onto the observed block,
\begin{equation}
\label{eq:projection}
\Pi_{\boldsymbol{X}}:\mathbb{R}^d\to\mathbb{R}^{d_x},
\qquad
(\boldsymbol{X},\boldsymbol{Y})\mapsto\boldsymbol{X},
\end{equation}
acts on a state-space coordinate and on the terminal point of a trajectory alike, since it is the same linear map on $\mathbb{R}^d$ in both cases.

Inference of the parameter $\boldsymbol{\alpha}$ is based on the observable marginal probability density distribution function at time $t$

\begin{equation}
\fXcondAwithArg
:=
\int_{\mathbb{R}^{d_y}} \fXYcondAwithArg \,d\boldsymbol{Y},
\qquad
(\boldsymbol{X},t)\in\mathbb{R}^{d_x}\times[0,T].
\label{eq:marginal}
\end{equation}
We assume that $f_{\boldsymbol x\mid\boldsymbol{\alpha}}((\cdot,t) | \boldsymbol{\alpha})$ is weakly differentiable with respect to $\boldsymbol{\alpha}$ and that
$\partial_{\boldsymbol{\alpha}} f_{\boldsymbol x | \alpha}((\cdot,t)\mid\boldsymbol{\alpha})
\in L^1(\mathbb{R}^{d_x})
$
for all $t\in[0,T]$ and $\boldsymbol{\alpha}\in\mathcal{A}$.
With the parameter-to-marginal forward map 
\begin{equation}
\mathcal{F}_t:\mathcal{A}\to L^1(\mathbb{R}^{d_x}),
\qquad
\mathcal{F}_t(\boldsymbol{\alpha}):= 
f_{\boldsymbol{x}| \boldsymbol{\alpha}}(\cdot,t\mid\boldsymbol{\alpha}).
\label{eq:forward_map}
\end{equation}
we can define the inverse problem as follows.

\begin{definition}[Inverse problem]
Given observed marginal probability density distribution function data 
$\left\{\tilde{f}_{\boldsymbol{x}}(\boldsymbol{X},T_m),  m=1 \ldots M\right\}$ at times $0<T_1<\cdots<T_m<\cdots < T_M$, the inverse problem is to find $\boldsymbol{\alpha}^\star\in\mathcal{A}$ such that
\begin{equation}
\label{eq:inverse_problem}
\boxed{
\mathcal{F}_{T_m}(\boldsymbol{\alpha}^\star)=\tilde{f}_{\boldsymbol{x}}(\boldsymbol{X},T_m),
\qquad m=1,\ldots,M.
}
\end{equation}
\end{definition}
An exact measurement of  ~\eqref{eq:marginal} is not feasible in practice. 
Because the marginal density is accessible only through finitely many empirical samples rather than an exact measurement, we compare distributions through their moments against a family of test functions, which remain well defined and stable under sampling noise. At time $T_m$, the moment with respect to a test function $\varphi_{k}$ is:
\begin{equation}
\tilde{\mu}_{k,m}:=\int_{\mathbb{R}^{d_x}}\varphi_{k}(\boldsymbol{X})\,\tilde{f}_{\boldsymbol{x}}(\boldsymbol{X},T_m)\,d\boldsymbol{X} ,
\end{equation}
where $\{\varphi_{k}\}_{k,1}^{K}\subset C_b(\mathbb{R}^{d_x})$  are continuous and bounded test functions (for instance, Hermite functions, compactly supported wavelets, or random Fourier features with growing bandwidth). The 
$\ell^2$ norm with respect to the forward model moment provides the following weak  
measure for the  data-misfit 
\begin{equation}
\label{eq:weak-loss}
\mathcal{J}_{K,M}(\boldsymbol{\alpha})
:=
\frac{1}{2}\sum_{m=1}^M\sum_{k=1}^K\bigl(\tilde{\mu}_{k,m}-\mu_{k,m}(\boldsymbol{\alpha})\bigr)^2.
\end{equation}
If there is only one measurement time ($M=1$), the short-hand notation to drop $m$ and $M$ is adopted.
An important explanation for the moment, in terms of expectation, is as follows,
\begin{equation}
\begin{aligned}
\label{eq:forward-moment}
\mu_{k,m}&:=\int_{\mathbb{R}^{d_x}}\varphi_{k}(\boldsymbol{X})\,f_{\boldsymbol{x}\mid\boldsymbol{\alpha}}(\boldsymbol{X},T_m \mid \boldsymbol{\alpha})\,d\boldsymbol{X}\\
 &=\int_{\mathbb{R}^{d_x}}\varphi_{k}(\boldsymbol{X})\,\int_{\mathbb{R}^{d_y}}\fXYcondA(\boldsymbol{X},\boldsymbol{Y},T_m \mid \boldsymbol{\alpha})\,d\boldsymbol{Y}d\boldsymbol{X}\\ &=\int_{\mathbb{R}^{d_x}}\int_{\mathbb{R}^{d_y}}\varphi_{k}(\boldsymbol{X})\,\fXYcondA(\boldsymbol{Q},T_m \mid \boldsymbol{\alpha})\,d\boldsymbol{X}d\boldsymbol{Y}\\
&=\int_{\mathbb{R}^{d}}\varphi_{k}(\Pi_{\boldsymbol{X}}\boldsymbol{Q})\,\fXYcondA(\boldsymbol{Q},T_m\mid \boldsymbol{\alpha})\,d\boldsymbol{Q}\\ &=\int_{\mathbb{R}^{d}}\varphi_{k}\!\left[\Pi_{\boldsymbol{X}}\Phi_{T_m}(\boldsymbol{Q}_0;\boldsymbol{\alpha})\right]\,f_0(\boldsymbol{Q}_0)\,d\boldsymbol{Q}_0\\ &=\int_{\mathbb{R}^{d}}\varphi_{k}\left[\Pi_{\boldsymbol{X}}\Phi_{T_m}(\boldsymbol{q}_0;\boldsymbol{\alpha})\right]\,f_0(\boldsymbol{q}_0)\,d\boldsymbol{q}_0 = \mathbb{E}_{\boldsymbol{q}_0\sim f_0}\!\Bigl[\varphi_k\bigl(\Pi_{\boldsymbol{X}}\Phi_{T_m}(\boldsymbol{q}_0;\boldsymbol{\alpha})\bigr)\Bigr]
\end{aligned}
\end{equation}
The operator $\Pi_{\boldsymbol{X}}$ projects the flow map onto the space defined by the observables $\boldsymbol{X}$. The expectation of the $k^{th}$ test function evaluated along the characteristic flow map provides the forward model evaluation of the weak moment $\tilde{\mu}_k(\boldsymbol{\alpha})$ that is compared to the weak loss in (\ref{eq:weak-loss}).
Based on this loss function,  we can define the inverse problem based on a weak loss as follows:

\begin{definition}[Inverse problem with Weak Loss]
Given the data of the observed marginal probability density distribution function 
$\left\{\tilde{f}_{\boldsymbol{x}}(\boldsymbol{X}, T_m),  m=1 \ldots M\right\}$ at times $0<T_1<\cdots<T_m<\cdots \le T_M$, the inverse problem is to find $\boldsymbol{\alpha}^\star\in\mathcal{A}$ such that
\begin{equation}
\label{eq:inverse_problem_weak}
\boxed{
\mathcal{J}_{K,M}(\boldsymbol{\alpha}^\star)=0,
\qquad M\ge 1
}
\end{equation}
with $\mathcal{J}_{K,M}(\boldsymbol{\alpha}^\star)$ the weak multi-time loss defined according to \eqref{eq:weak-loss} determined by the $\ell_2$ norm of the difference between the weak loss \eqref{eq:weak-loss} and the weak forward model moment  \eqref{eq:forward-moment}.
\end{definition}
This replaces the requirement that the predicted and observed marginals agree at every point $\boldsymbol{X}$, the inverse problem~\eqref{eq:inverse_problem}, by a finite-dimensional misfit on $K$ summary statistics for $M$ time instances. 

In general, the inverse problem~\eqref{eq:inverse_problem} is ill-posed, since marginalization over the latent coordinates discards information from the joint density. Hyperbolicity together with the characteristic representation nonetheless restores local identifiability of $\boldsymbol{\alpha}^\star$: the flow along forward and backward characteristics yields an unbiased, computable sensitivity of the observed marginal to $\boldsymbol{\alpha}$, and $\boldsymbol{\alpha}^\star$ is locally identifiable precisely if the resulting sensitivity Gramian is positive definite, as made precise in the identifiability analysis below.
The sensitivity derived in the next section is instrumental in a stable iterative convergence to $\boldsymbol{\alpha}^\star$.

\subsection{Parameter sensitivity of the characteristic solution}
\label{subsec:backward_scan}


In the Liouville equation, the density is described in an Eulerian form: at each
time \(T\), \(\fXYcondA(\boldsymbol Q,T\mid\boldsymbol\alpha)\) is a function on
the state space, evaluated at a sample-space coordinate \(\boldsymbol Q\). To
study the parameter sensitivity of this Eulerian density, we therefore fix a
terminal state-space point \(\boldsymbol Q_T\) and differentiate the value of the
density at that point with respect to \(\boldsymbol\alpha\). To emphasize this
Eulerian convention, we write the derivative as \(\dpEul{\,\cdot\,}\). This is distinct from a Lagrangian sensitivity, where a sample is labeled by its initial location $\boldsymbol{q}_0$, and the derivative is taken along a characteristic
trajectory initialized at the fixed \(\boldsymbol q_0\), as illustrated in figure \ref{fig:sensitivity_schematic}.

\begin{definition}[Eulerian parameter sensitivity]
 \label{def:G}
For a terminal state-space point \(\boldsymbol Q_T\in\mathbb{R}^d\),
a final time \(T>0\), and a parameter \(\boldsymbol\alpha\in\mathcal A\), define 
\begin{equation}
\label{eq:G-def}
G(\boldsymbol{Q}_T,T;\boldsymbol{\alpha})
:=
\dpEul{\log \fXYcondA(\boldsymbol{Q}_T,T\mid\boldsymbol{\alpha})}
\in\mathbb{R}^p
\end{equation}
denotes the Eulerian sensitivity to $\alpha$ of the log-density of the forward solution at the Eulerian point $\boldsymbol{Q}_T$.
Under Assumption~\ref{ass:char_regularity}, the right-hand side of~\eqref{eq:char-solution-terminal} is $C^1$ in $\boldsymbol{\alpha}$, so $G$ is well-defined.
\end{definition}
The direct differentiation of the derivative in $\dpEul{\log \fXYcondA(\boldsymbol{Q}_T,T\mid\boldsymbol{\alpha})} $ based on the characteristic solution in ~\eqref{eq:char-solution-terminal} requires the derivatives of $\Phi_T$ and of $\Phi_T^{-1}$ with respect to $\boldsymbol{\alpha}$. These flow-map Jacobians are available in principle but costly to form directly, since they entail differentiating the entire flow map. The solution along the backward-time characteristic offers an alternative that determines the sensitivity without direct differentiation of the flow map. 

Along this characteristic, we introduce the density factor \(
\rho(t)
:=
\frac{
\fXYcondA(\boldsymbol Q_T,T\mid\boldsymbol{\alpha})
}{\fXYcondA(\boldsymbol q(t),t\mid\boldsymbol{\alpha})
}.
\)
Thus \(\rho(T)=1\), or equivalently \(\log\rho(T)=0\).
To find a backward-time solution, we start by considering  the terminal value problem,
\begin{equation}
\label{eq:backward-scan}
\frac{d\boldsymbol{q}}{dt}=\boldsymbol{R}(\boldsymbol{q},t;\boldsymbol{\alpha}),
\qquad
\frac{d\log\rho}{dt}=h_{\boldsymbol{\alpha}}(\boldsymbol{q},t),
\qquad
\boldsymbol{q}(T)=\boldsymbol{Q}_T,\quad\log\rho(T)=0,
\end{equation}
integrated from $t=T$ down to $t=0$.
By uniqueness of $\Phi_T$, $\boldsymbol{q}(0)=\Phi_T^{-1}(\boldsymbol{Q}_T;\boldsymbol{\alpha})$, and direct integration gives $\log\rho(0)=-\int_0^T h_{\boldsymbol{\alpha}}(\boldsymbol{q}(s),s)\,ds$; substituting into Lemma~\ref{lem:char_representation} yields
\begin{equation}
\label{eq:log-f-from-backward}
\log f(\boldsymbol{Q}_T,T;\boldsymbol{\alpha})=\log f_0\bigl(\boldsymbol{q}(0)\bigr)+\log\rho(0).
\end{equation}
Although the sensitivity is defined at a terminal Eulerian point
\(\boldsymbol Q_T\), the points at which it is evaluated are generated
Lagrangianly, given an initial sample \(\boldsymbol q_0\sim f_0\).
Hence \(\boldsymbol q_0\) labels the tracer, while \(\boldsymbol Q_T\) denotes
the state-space location at which that tracer arrives at time \(T\).
The unique characteristic
of~\eqref{eq:backward-scan} associated with $(\boldsymbol{Q}_T,\boldsymbol{\alpha})$.
is  denoted as $\boldsymbol{q}(\cdot)=\boldsymbol{q}(\cdot;\boldsymbol{Q}_T,\boldsymbol{\alpha})$.

\begin{figure}[h!]
    \centering
    \includegraphics[width=\linewidth]{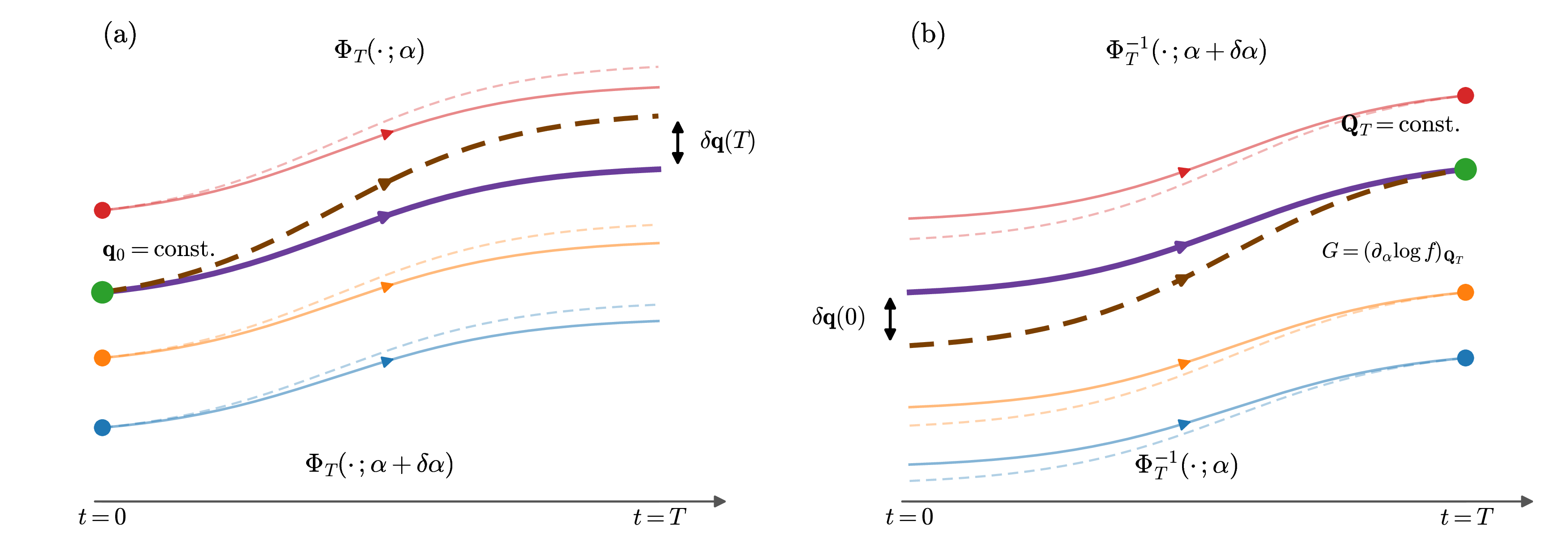}
    \caption{The two parameter-sensitivity anchorings, shown on a bundle of
    characteristics with one tracer highlighted, as $\boldsymbol{\alpha}$ varies from
    $\boldsymbol{\alpha}^\star$ to $\boldsymbol{\alpha}^\star+\delta\boldsymbol{\alpha}$;
    solid curves are the flow $\Phi_T(\cdot\,;\boldsymbol{\alpha})$ and dashed curves the flow
    at $\boldsymbol{\alpha}+\delta\boldsymbol{\alpha}$.
    $(a)$~the initial point $\boldsymbol{q}_0$ is held fixed and shared by both flows, and the
    terminal points differ by $\delta\boldsymbol{q}(T)$.
    $(b)$~the terminal Eulerian point $\boldsymbol{Q}_T$ is held fixed and shared by both
    flows, the initial points differ by $\delta\boldsymbol{q}(0)$, and the density sensitivity
    $G$ is read at $\boldsymbol{Q}_T$.
    Filled dots mark the held-fixed point of the highlighted tracer.}
    \label{fig:sensitivity_schematic}
\end{figure}  

The following theorem shows that the Eulerian sensitivity depends uniquely on the sensitivity of the accumulated divergence with respect to $\boldsymbol{\alpha}$ along the backward-time characteristic, which prevents the need for direct differentiation of the flow map.

\begin{theorem}[Eulerian sensitivity along backward characteristics]
\label{thm:G_formula}
The Eulerian sensitivity decomposes as 
\begin{equation}
G(\boldsymbol{Q}_T,T;\boldsymbol{\alpha})=S_{\boldsymbol{\alpha}}(0)^{\!\top}\nabla\log f_0(\boldsymbol{q}(0))+r(T),
\label{eqn:G-decomposition}
\end{equation}
where $r(T)$
is the sensitivity of the accumulated divergence with respect to $\boldsymbol{\alpha}$,
\begin{equation}
r(T):=\dpEul{\log\rho(0)}=-\int_0^T[\partial_{\boldsymbol{\alpha}}h_{\boldsymbol{\alpha}}+(\nabla_{\boldsymbol{q}}h_{\boldsymbol{\alpha}})^{\!\top}S_{\boldsymbol{\alpha}}(t)]\,dt\in\mathbb{R}^p,
\label{eqn:r-formula}
\end{equation}
with  $S_{\boldsymbol{\alpha}}(t):=\dpEul{\boldsymbol{q}(t;\boldsymbol{Q}_T,\boldsymbol{\alpha})}\in\mathbb{R}^{d\times p}$
the unique solution of the equation
\begin{equation}
\label{eq:backward-sensitivity}
\dot S_{\boldsymbol{\alpha}}(t)
=
\nabla_{\boldsymbol{q}}\boldsymbol{R}\,S_{\boldsymbol{\alpha}}(t)
+\partial_{\boldsymbol{\alpha}}\boldsymbol{R},
\qquad
S_{\boldsymbol{\alpha}}(T)=\boldsymbol{0},
\end{equation}
integrated from $t=T$ to $t=0$, with unindicated arguments of $(\boldsymbol{q}(t),t)$.
\end{theorem}

\begin{proof}
The terminal condition $S_{\boldsymbol{\alpha}}(T)=\boldsymbol{0}$ is immediate from $\boldsymbol{q}(T)=\boldsymbol{Q}_T$ being held fixed by definition of $S_{\boldsymbol{\alpha}}$.
Differentiating the state equation $d\boldsymbol{q}/dt=\boldsymbol{R}(\boldsymbol{q}(t),t;\boldsymbol{\alpha})$ with respect to $\boldsymbol{\alpha}$ at fixed $\boldsymbol{Q}_T$ results in~\eqref{eq:backward-sensitivity}.

From~\eqref{eq:backward-scan}, $\log\rho(0)=-\int_0^T h_{\boldsymbol{\alpha}}(\boldsymbol{q}(t),t)\,dt$.
The integrand is $C^1$ in $(\boldsymbol{\alpha},t)$ on the compact $[0,T]$ with bounded $\boldsymbol{\alpha}$-derivative, so differentiation under the integral applies and
\[
r(T)
=
-\int_0^T
\Bigl[\partial_{\boldsymbol{\alpha}}h_{\boldsymbol{\alpha}}(\boldsymbol{q}(t),t)
+\nabla_{\boldsymbol{q}}h_{\boldsymbol{\alpha}}(\boldsymbol{q}(t),t)^{\!\top}\,\dpEul{\boldsymbol{q}(t)}\Bigr]\,dt,
\]
which is \eqref{eqn:r-formula} of Theorem~\ref{thm:G_formula} after substituting $\dpEul{\boldsymbol{q}(t)}=S_{\boldsymbol{\alpha}}(t)$.

Finally, differentiating $\log f(\boldsymbol{Q}_T,T;\boldsymbol{\alpha})=\log f_0(\boldsymbol{q}(0))+\log\rho(0)$ with respect to $\boldsymbol{\alpha}$ at fixed $\boldsymbol{Q}_T$ and applying the chain rule yields
\[
G(\boldsymbol{Q}_T,T;\boldsymbol{\alpha})
=
\nabla\log f_0(\boldsymbol{q}(0))^{\!\top}\,\partial_{\boldsymbol{\alpha}}\boldsymbol{q}(0)
+
\partial_{\boldsymbol{\alpha}}\log\rho(0)
=
S_{\boldsymbol{\alpha}}(0)^{\!\top}\,\nabla\log f_0(\boldsymbol{q}(0))+r(T),
\]
which is the decomposition~\eqref{eqn:G-decomposition} of $G$ in Theorem~\ref{thm:G_formula}.
\end{proof}

\begin{remark}[Uniform initial density]
\label{rem:uniform_f0}
 Under Assumption~\ref{ass:char_regularity}, the right-hand side of~\eqref{eq:char-solution-terminal} is $C^1$ in $\boldsymbol{\alpha}$ for terminal points whose backward characteristic starts in $\Omega^{\!\circ}$. 
For uniform $f_0$, $\nabla\log f_0\equiv\boldsymbol{0}$ in $\Omega^{\!\circ}$,so
the first term in the decomposed expression for $G$ vanishes, and thus $G(\boldsymbol{Q}_T,T;\boldsymbol{\alpha})=r(T)$.
The boundary $\partial\Omega$ has $f_0$-measure zero, so its contributions vanish in a weak sense.
\end{remark}


\begin{remark}[Implementation using automatic-differentiation]
\label{rem:ad_implementation}
In implementations, the system to compute sensitivity in Theorem~\ref{thm:G_formula} is not formed explicitly. Instead, the system~\eqref{eq:backward-scan} is integrated backward in time by a differentiable time-stepping scheme, resulting in $\log f(\boldsymbol{Q}_T,T;\boldsymbol{\alpha})$ as a $C^1$ function of $\boldsymbol{\alpha}$ with $\boldsymbol{Q}_T$ treated as constant.
Forward-mode automatic differentiation then yields $G\in\mathbb{R}^p$ in a single backward pass at cost proportional to $p$ times one forward solve; equivalently, the auto-diff trace propagates $S_{\boldsymbol{\alpha}}$ and $r$ through the same time-stepping scheme, without ever forming the $d\times p$ sensitivity matrix.
\end{remark}

\begin{remark}[Relation to adjoint methods for trajectory losses]
\label{rem:adjoint_vs_density}
The backward system~\eqref{eq:backward-scan} resembles adjoint methods for trajectory-dependent losses~\cite{chen2018neural,dominguez2022inference}, in that both integrate an auxiliary variable backward along a characteristic.
The two differ in objective and in the structural role of the divergence.
Standard adjoint methods compute $\partial_{\boldsymbol{\alpha}}\mathcal{L}(\boldsymbol{q}_T)$ for a scalar loss $\mathcal{L}$ that depends only on the terminal state, and the divergence $h_{\boldsymbol{\alpha}}$ does not appear.
Theorem~\ref{thm:G_formula} computes instead $\partial_{\boldsymbol{\alpha}}\log f(\boldsymbol{Q}_T,T;\boldsymbol{\alpha})$ at a fixed terminal point; this requires both the propagation of $\nabla\log f_0$ along the inverse flow and the parametric sensitivity of the accumulated divergence.
The latter term encodes local volume change of the flow and has no analog in trajectory-based adjoint formulations; it is precisely the term that distinguishes Liouville log-density sensitivities from standard ODE parameter sensitivities.
\end{remark}

\subsection{Stochastic gradient}
\label{subsec:weak_loss}

In order to iterate to a converged parameter, an unbiased stochastic gradient of the weak loss function is constructed by combining the Eulerian sensitivity 
in (\ref{eqn:G-decomposition})
with a crossed U-statistic that correlates the weak loss function for a set of multiple characteristics.

\begin{lemma}[Gradient of weak-loss function] 
\label{lem:score_identity}
Under Assumption~\ref{ass:char_regularity}, for a single measurement time $T$, the boundedness of the test functions, and the local integrability condition,
\[
\mathbb{E}_{\boldsymbol{q}_0\sim f_0}\|G(\Phi_T(\boldsymbol{q}_0;\boldsymbol{\alpha}),T;\boldsymbol{\alpha})\|<\infty,
\]
the gradient of the weak loss is expressed as
\begin{equation}
\label{eq:grad-weak-loss}
\partial_{\boldsymbol{\alpha}}\mathcal{J}_K(\boldsymbol{\alpha})
=
\sum_{k=1}^K\bigl(\mu_k(\boldsymbol{\alpha})-\tilde{\mu}_k\bigr)\,\partial_{\boldsymbol{\alpha}}\mu_k(\boldsymbol{\alpha}),
\end{equation}
where 
the feature gradient satisfies,
\begin{equation}
\label{eq:feature-sensitivity}
\partial_{\boldsymbol{\alpha}}\mu_k(\boldsymbol{\alpha})
=
\mathbb{E}_{\boldsymbol{q}_0\sim f_0}\!\Bigl[\varphi_k\bigl(\Pi_{\boldsymbol{X}}\boldsymbol{q}_T\bigr)\,G(\boldsymbol{q}_T,T;\boldsymbol{\alpha})\Bigr],
\qquad
\boldsymbol{q}_T:=\Phi_T(\boldsymbol{q}_0;\boldsymbol{\alpha}),
\end{equation}
and the gradient of the weak loss is therefore
a sum of products of expectations.
\end{lemma}

Here \(\boldsymbol q_T\) denotes the terminal position of the Lagrangian tracer initialized at \(\boldsymbol q_0\). Using it as the Eulerian argument of \(G\) means the state-space point reached by the
tracer, namely \(\boldsymbol Q_T=\boldsymbol q_T\).

\begin{proof}
Push the expectation in $\mu_k$ from $\boldsymbol{q}_0$ to the terminal point via the flow $\Phi_T(\cdot;\boldsymbol{\alpha})$, with $f(\boldsymbol{Q}_T,T;\boldsymbol{\alpha})$ as the forward-time moment:
\[
\mu_k(\boldsymbol{\alpha})
=
\int_{\mathbb{R}^d}\varphi_k\bigl(\Pi_{\boldsymbol{X}}\boldsymbol{Q}_T\bigr)\,f(\boldsymbol{Q}_T,T;\boldsymbol{\alpha})\,d\boldsymbol{Q}_T.
\]
The local integrability of $G$ provides an integrable dominating function for the differentiated probability density function, since $\partial_{\boldsymbol{\alpha}}f=fG$ at a fixed final time, $T$,  yielding
\[
\begin{aligned}
\partial_{\boldsymbol{\alpha}}\mu_k(\boldsymbol{\alpha})
&=
\int_{\mathbb{R}^d}\varphi_k\bigl(\Pi_{\boldsymbol{X}}\boldsymbol{Q}_T\bigr)\,\partial_{\boldsymbol{\alpha}}f(\boldsymbol{Q}_T,T;\boldsymbol{\alpha})\,d\boldsymbol{Q}_T\\
&=
\int_{\mathbb{R}^d}\varphi_k\bigl(\Pi_{\boldsymbol{X}}\boldsymbol{Q}_T\bigr)\,f(\boldsymbol{Q}_T,T;\boldsymbol{\alpha})\,G(\boldsymbol{Q}_T,T;\boldsymbol{\alpha})\,d\boldsymbol{Q}_T,
\end{aligned}
\]
where the second equality uses $\partial_{\boldsymbol{\alpha}}f=f\,G$ from Definition~\ref{def:G}.
Reverting the expectation to $\boldsymbol{q}_0\sim f_0$ with $f(\boldsymbol{Q}_T,T;\boldsymbol{\alpha})\,d\boldsymbol{Q}_T=f_0(\boldsymbol{Q}_0)\,d\boldsymbol{Q}_0$ gives~\eqref{eq:feature-sensitivity}; the chain rule applied to $\mathcal{J}_K$ then yields~\eqref{eq:grad-weak-loss}.
\end{proof}

The right-hand side of~\eqref{eq:grad-weak-loss} is a sum of products of two expectations: the data misfit $\mu_k-\tilde{\mu}_k$ and the gradient $\partial_{\boldsymbol{\alpha}}\mu_k$.
A naive Monte Carlo estimator that reuses one 
solution manifold for both factors is biased and of order $1/N$, where $N$ is the number Monte-Carlo sample manifolds \cite{lee2019unbiased}.
To prevent this bias, we use the crossed U-statistic \cite{hoeffding1948ustatistic}, in which the weak loss and forward model moment are evaluated on disjoint sample manifolds.

\begin{proposition}[Unbiased crossed estimator]
\label{prop:crossed_ustatistic}
Let $\{\boldsymbol{q}_0^{(n)}\}_{n=1}^N\sim f_0$ be i.i.d.\ with $N\geq 2$, set $\boldsymbol{q}_T^{(n)}:=\Phi_T(\boldsymbol{q}_0^{(n)};\boldsymbol{\alpha})$, $\mu_k^{(n)}:=\varphi_k(\Pi_{\boldsymbol{X}}\boldsymbol{q}_T^{(n)})$, and $G^{(n)}:=G(\boldsymbol{q}_T^{(n)},T;\boldsymbol{\alpha})$.
Under Assumption~\ref{ass:char_regularity} and the boundedness of the test functions, the crossed estimator
\begin{equation}
\label{eq:crossed-estimator}
\widehat{\boldsymbol{g}}=
\widehat{\partial_{\boldsymbol{\alpha}}\mathcal{J}_K}
:=
\sum_{k=1}^K\frac{1}{N(N-1)}\sum_{1\leq i\neq j\leq N}\bigl(\mu_k^{(i)}-\tilde{\mu}_k\bigr)\,\mu_k^{(j)}\,G^{(j)}
\end{equation}
is an unbiased estimator of the gradient~\eqref{eq:grad-weak-loss}: $\mathbb{E}\bigl[\widehat{\partial_{\boldsymbol{\alpha}}\mathcal{J}_K}\bigr]=\partial_{\boldsymbol{\alpha}}\mathcal{J}_K(\boldsymbol{\alpha})$.
\end{proposition}

\begin{proof}
For $i\neq j$ the particles $\boldsymbol{q}_0^{(i)}$ and $\boldsymbol{q}_0^{(j)}$ are independent, so for each $k$
\[
\mathbb{E}\bigl[(\mu_k^{(i)}-\tilde{\mu}_k)\,\mu_k^{(j)}G^{(j)}\bigr]
=
\mathbb{E}\bigl[\mu_k^{(i)}-\tilde{\mu}_k\bigr]\,\mathbb{E}\bigl[\mu_k^{(j)}G^{(j)}\bigr]
=
\bigl(\mu_k(\boldsymbol{\alpha})-\tilde{\mu}_k\bigr)\,\partial_{\boldsymbol{\alpha}}\mu_k(\boldsymbol{\alpha}),
\]
where the last identity uses Lemma~\ref{lem:score_identity}.
Averaging over the $N(N-1)$ ordered pairs of solution manifolds and summing over $k$ test functions yields ~\eqref{eq:grad-weak-loss}.
\end{proof}

\begin{remark}[Extension to multiple observation times]
The implementation for the weak-loss function for multiple times $\mathcal{J}_{K,M}$ in (\ref{eq:weak-loss}) reuses a single forward set of $N$ solution manifolds.
Initial particles $\{\boldsymbol{q}_0^{(n)}\}_{n=1}^N\sim f_0$ are integrated once on $[0,T_M]$, recording snapshots $\boldsymbol{q}^{(n)}(T_m)$ at each observation time.
At each $T_m$ the snapshot is frozen and the backward scan~\eqref{eq:backward-scan} is run from $T_m$ back to $0$ to produce $G^{(n)}(T_m):=G(\boldsymbol{q}^{(n)}(T_m),T_m;\boldsymbol{\alpha})$.
Conditional on the frozen snapshots, the $M$ backward scans are independent and can be concatenated into a single vector-valued function differentiated by one forward-mode automatic differentiation, resulting in the $MN\times p$ Jacobian of all snapshot features in a single pass.
Inserting the extra time indices $m$ into the features and gradients of~\eqref{eq:crossed-estimator} yields
\begin{equation}
\label{eq:multi_time_estimator}
\widehat{\boldsymbol{g}}=\widehat{\partial_{\boldsymbol{\alpha}}\mathcal{J}_{K,M}}
:=
\sum_{m=1}^M\sum_{k=1}^K\frac{1}{N(N-1)}\sum_{1\leq i\neq j\leq N}\bigl(\mu_{k,m}^{(i)}-\tilde{\mu}_{k,m}\bigr)\,\mu_{k,m}^{(j)}\,G^{(j)}(T_m),
\end{equation}
which inherits unbiasedness from Proposition~\ref{prop:crossed_ustatistic}.
The snapshots at different times are generally correlated because they stem from the same initial conditions.
\end{remark}

\subsection{Stochastic gradient descent (SGD) with full-matrix preconditioning}
\label{subsec:preconditioning}

The crossed estimator $\widehat{\boldsymbol{g}}\in\mathbb{R}^p$ of~\eqref{eq:crossed-estimator}, or its multi-time analogue~\eqref{eq:multi_time_estimator}, drives the preconditioned stochastic gradient descent (SGD) iteration
\begin{equation}
\label{eq:sgd_update}
\boldsymbol{\alpha}_{\ell+1}
=
\boldsymbol{\alpha}_\ell-\eta_\ell\,H_\ell^{-1/2}\,\widehat{\boldsymbol{g}}_\ell,
\end{equation}
where $\ell$ is the iteration number, and $H_\ell\succ\boldsymbol{0}$ is a positive-definite preconditioner.
For moderate parameter dimensions ($p\lesssim 10^2$), the cross-component correlations of the gradient induced by the small set of characteristics in~\eqref{eq:crossed-estimator} are non-negligible.
Diagonal preconditioners such as Adam~\cite{kingma2014adam}, which maintain only per-component second moments $\mathbb{E}[\widehat{g}_{\ell,j}^{\,2}]$, ignore this structure and converge slowly along the anisotropic directions of the loss.
A full-matrix preconditioner is affordable in this regime.
The Shampoo update~\cite{gupta2018shampoo} maintains the $p\times p$ running second-moment matrix
\begin{equation}
\label{eq:shampoo-update}
H_\ell
=
\beta\,H_{\ell-1}+(1-\beta)\,\widehat{\boldsymbol{g}}_\ell\widehat{\boldsymbol{g}}_\ell^{\top},
\qquad
H_0=\varepsilon\,I_p,
\end{equation}
with $\beta\in(0,1)$ a forgetting factor and $\varepsilon>0$ a damping constant.

The matrix inverse square root $H_\ell^{-1/2}$ is computed by eigendecomposition at $O(p^3)$ cost per step, negligible compared with the $O(N\,T/\Delta t)$ cost of the forward solves and backward scans of one set of characteristics.
The matrix $H_\ell$ tracks the running second moment of the stochastic gradient $\mathbb{E}\bigl[\widehat{\boldsymbol{g}}_\ell\widehat{\boldsymbol{g}}_\ell^{\!\top}\bigr]$ rather than the Hessian itself. 
Thus, $H_\ell$ should be interpreted as an adaptive gradient-covariance preconditioner. It can substantially improve conditioning if the dominant gradient covariance directions align with the anisotropic directions of the loss.

The step-size sequence $\{\eta_\ell\}$ satisfies the Robbins-Monro conditions $\sum_\ell\eta_\ell=\infty$ and $\sum_\ell\eta_\ell^2<\infty$~\cite{robbins1951stochastic}. 
Together with unbiased gradients, bounded second moments, a Lipschitz gradient on the relevant sublevel set, and uniformly bounded positive-definite preconditioners, these conditions give the standard stochastic-approximation convergence result: $\liminf_{\ell\to\infty}\|\partial_{\boldsymbol{\alpha}}\mathcal{J}_K(\boldsymbol{\alpha}_\ell)\|=0$ almost surely~\cite{bottou2018optimization}.
In implementations the inverse-time schedule $\eta_\ell=\eta_0/(1+\ell/\tau)$ is used.

\subsection{Local identifiability via linearized sensitivity}
\label{subsec:local_identifiability}

The ability to identify the true parameter $\boldsymbol{\alpha}^\star \in \mathcal{A}$ from the feature moments, i.e., the identifiability, as well as the convergence rate of the stochastic gradient descent iteration~\eqref{eq:sgd_update}, are both determined by the local sensitivity of the flow map  $\Phi_t(\boldsymbol{q}_0;\boldsymbol{\alpha})\Big|_{\boldsymbol{\alpha}^\star}$  in a neighborhood of $\boldsymbol{\alpha}^\star$.
The sensitivity of the solution along its Lagrangian path in the vicinity of $\boldsymbol{\alpha}^\star$, with a fixed initial condition,
\begin{equation}
\label{eq:forward_sens_def2}
S_{\boldsymbol{q}_0}(t)
:=
\dpLag{\Phi_t(\boldsymbol{q}_0;\boldsymbol{\alpha})}\big|_{\boldsymbol{\alpha}^\star}
\in\mathbb{R}^{d\times p},
\end{equation}
can be interpreted as the forward-anchored counterpart of the backward sensitivity $S_{\boldsymbol{\alpha}}(t)$ 
with a fixed end state $\boldsymbol{Q}_T$.
It is governed by the linearized equation
\begin{equation}
\label{eq:forward_sensitivity_ode}
\dot{S}_{\boldsymbol{q}_0}(t)
=
\nabla_{\boldsymbol{q}}\boldsymbol{R}\bigl(\boldsymbol{q}^\star(t),t;\boldsymbol{\alpha}^\star\bigr)\,S_{\boldsymbol{q}_0}(t)
+
\partial_{\boldsymbol{\alpha}}\boldsymbol{R}\bigl(\boldsymbol{q}^\star(t),t;\boldsymbol{\alpha}^\star\bigr),
\qquad
S_{\boldsymbol{q}_0}(0)=\boldsymbol{0}.
\end{equation}
Let $\Psi(t,s)\in\mathbb{R}^{d\times d}$ denote the state-transition matrix of the linear time-varying system $\dot{\boldsymbol{v}}=\nabla_{\boldsymbol{q}}\boldsymbol{R}\bigl(\boldsymbol{q}^\star(t),t;\boldsymbol{\alpha}^\star\bigr)\,\boldsymbol{v}$, characterized by $\partial_t\Psi(t,s)=\nabla_{\boldsymbol{q}}\boldsymbol{R}\bigl(\boldsymbol{q}^\star(t),t;\boldsymbol{\alpha}^\star\bigr)\,\Psi(t,s)$ and $\Psi(s,s)=I_d$.
Duhamel's principle yields the closed-form expression
\begin{equation}
\label{eq:forward_sens_duhamel}
S_{\boldsymbol{q}_0}(T)
=
\int_0^T\Psi(T,s)\,\partial_{\boldsymbol{\alpha}}\boldsymbol{R}\bigl(\boldsymbol{q}^\star(s),s;\boldsymbol{\alpha}^\star\bigr)\,ds,
\end{equation}
which represents the terminal sensitivity as parametric forcing propagated forward by the linearized dynamics over $[s,T]$.

With $M$ observation times $\{T_m\}_{m=1}^M$ and $K$ test functions $\{\varphi_k\}_{k=1}^K$, the sensitivity Jacobian of the multi-time weak loss~\eqref{eq:weak-loss} is the $MK\times p$ matrix with entries
\begin{equation}
\label{eq:sens_jacobian}
\mathcal{S}_{(m,k),j}
:=
\partial_{\alpha_j}\mu_{k,m}(\boldsymbol{\alpha})\Big|_{\boldsymbol{\alpha}^\star}
=
\mathbb{E}_{\boldsymbol{q}_0\sim f_0}\!\Bigl[\nabla_{\boldsymbol{x}}\varphi_k\bigl(\Pi_{\boldsymbol{X}}\boldsymbol{q}^\star(T_m)\bigr)^{\!\top}\,\Pi_{\boldsymbol{X}}\,S_{\boldsymbol{q}_0}(T_m)\Bigr]_j.
\end{equation}
The associated feature-sensitivity Gramian is the $p\times p$ Gram matrix
\begin{equation}
\label{eq:fisher_information}
\mathcal{I}_T
:=
\mathcal{S}^{\!\top}\mathcal{S}.
\end{equation}

A sufficient condition for local identifiability of $\boldsymbol{\alpha}^\star$ from the moments $\mu_{k,m}$ is the full-column-rank condition $\mathcal{I}_T\succ\boldsymbol{0}$.
Under this condition, Taylor expansion of $\mu_{k,m}$ around $\boldsymbol{\alpha}^\star$ with $\mu_{k,m} (\boldsymbol{\alpha}^\star)=\tilde{\mu}_{k,m}$ gives the following local quadratic relation
\begin{equation}
\label{eq:param_quadratic}
\mathcal{J}_{K,M}(\boldsymbol{\alpha})
=
\tfrac{1}{2}\,(\boldsymbol{\alpha}-\boldsymbol{\alpha}^\star)^{\!\top}\,\mathcal{I}_T\,(\boldsymbol{\alpha}-\boldsymbol{\alpha}^\star)
+
o\bigl(\|\boldsymbol{\alpha}-\boldsymbol{\alpha}^\star\|^2\bigr),
\end{equation}
so $\boldsymbol{\alpha}^\star$ is an isolated minimizer of the multi-time weak loss~\eqref{eq:weak-loss}.

Note that two structural mechanisms degrade the conditioning of $\mathcal{I}_T$ and consequently the parameter-recovery rate of the SGD iteration.
Firstly, With $L_R$ the global Lipschitz constant of Assumption~\ref{ass:char_regularity}, the fundamental matrix satisfies $\|\Psi(T,s)\|\leq e^{L_R(T-s)}$, so parametric forcing at early times is amplified by at most $e^{L_R T}$.
In chaotic regimes the propagator grows like $e^{\lambda_{\max}(T-s)}$ with $\lambda_{\max}$ the leading Lyapunov exponent of the linearized flow; both the signal $\mathcal{S}$ and the gradient noise \del{of Theorem~\ref{thm:lip-variance} }share this rate, and the signal-to-noise ratio saturates, which places a practical ceiling on the usable time horizon.
Secondly, marginalization over the latent coordinates restricts the columns of $\mathcal{S}$ to the projected sensitivity $\Pi_{\boldsymbol{X}}\,S_{\boldsymbol{q}_0}(T_m)$, so any direction $\boldsymbol{v}\in\mathbb{R}^p$ for which $\Pi_{\boldsymbol{X}}\,S_{\boldsymbol{q}_0}(T_m)\,\boldsymbol{v}\approx\boldsymbol{0}$ uniformly in $(m,\boldsymbol{q}_0)$ is effectively unidentifiable from the marginal observations.
\subsection{Pseudo Algorithm}
The method, combining the backward-characteristic,  Eulerian sensitivity, the crossed U-statistic gradient, and the Shampoo step, is summarized in Algorithm~\ref{alg:crossed-sgd}.

\begin{algorithm}[h!]
\caption{Parameter inference of ODES based on characteristic Liouville equations}
\label{alg:crossed-sgd}
\begin{algorithmic}[1]
\State \textbf{Inputs:} initial parameter $\boldsymbol{\alpha}_0$; step sizes $\{\eta_\ell\}_{\ell\ge0}$ with $\eta_\ell=\eta_0/(1+\ell/\tau)$; horizon $T>0$; test functions $\{\varphi_k\}_{k=1}^K\subset C_b(\mathbb{R}^{d_x})$ with measured features $\tilde{\mu}_{k}=\int\varphi_k(\boldsymbol{X})\,\tilde{f}_{\boldsymbol{X}}(\boldsymbol{X},T)\,d\boldsymbol{X}$; number of sampled characteristics $N\ge2$; initial density $f_0$; flow $\Phi_t(\cdot;\boldsymbol{\alpha})$ and divergence $h_{\boldsymbol{\alpha}}=\nabla_{\boldsymbol{q}}\!\cdot\boldsymbol{R}$.
\State \textbf{Output:} iterates $\{\boldsymbol{\alpha}_\ell\}$ approaching a stationary point of $\mathcal{J}_{K}$.
\State Initialize $H_0\gets\varepsilon I_p$.
\For{$\ell=0,1,2,\dots$}
  \State Draw $\{\boldsymbol{q}_0^{(n)}\}_{n=1}^N\stackrel{\mathrm{iid}}{\sim} f_0$.
  \For{$n=1,\dots,N$}
    \State Integrate $\dot{\boldsymbol{q}}=\boldsymbol{R}(\boldsymbol{q},t;\boldsymbol{\alpha}_\ell)$ forward from $\boldsymbol{q}_0^{(n)}$; record $\boldsymbol{q}_T^{(n)}$.
    \State Backward-integrate the sensitivity~\eqref{eq:backward-sensitivity} on $[T,0]$ to obtain $S_{\boldsymbol{\alpha}_\ell}^{(n)}(0)$; evaluate $r^{(n)}(T)$ via Theorem~\ref{thm:G_formula}.
    \State Assemble the Eulerian sensitivity $G^{(n)}\gets r^{(n)}(T)+S_{\boldsymbol{\alpha}_\ell}^{(n)}(0)^{\!\top}\nabla\log f_0(\boldsymbol{q}_0^{(n)})$ via Theorem~\ref{thm:G_formula}.
    \State $\mu_k^{(n)}\gets\varphi_k(\Pi_{\boldsymbol{X}}\boldsymbol{q}_T^{(n)})$ for $k=1,\dots,K$.
  \EndFor
  \State Crossed U-statistic gradient:
  \[
  \widehat{\boldsymbol{g}}_\ell
  \;\gets\;
  \sum_{k=1}^K\frac{1}{N(N-1)}\sum_{i\neq j}\bigl(\mu_k^{(i)}-\tilde{\mu}_{k}\bigr)\,\mu_k^{(j)}\,G^{(j)}.
  \]
  \State Shampoo preconditioner:
  \[
  H_{\ell+1}\gets\beta H_{\ell}+(1-\beta)\,\widehat{\boldsymbol{g}}_\ell\widehat{\boldsymbol{g}}_\ell^{\!\top}.
  \]
  \State Preconditioned step:
  \[
  \boldsymbol{\alpha}_{\ell+1}
  \gets
  \boldsymbol{\alpha}_\ell-\eta_\ell H_{\ell+1}^{-1/2}\widehat{\boldsymbol{g}}_\ell.
  \]
\EndFor
\end{algorithmic}
\end{algorithm}

\section{Convergence, Stability, and Accuracy}
\label{sec:algorithm}

\del{The convergence behavior of \ref{alg:crossed-sgd} is
separated into the stability of matching the model parameters, and the accuracy of matching the observations.
The mean-square parameter error and the weak loss both decay as $O(1/(N\ell))$~\eqref{eq:msa_rate}, so the
parameter error falls as $\ell^{-1/2}$ and the weak loss as $\ell^{-1}$, each improving as
$N^{-1/2}$ in the number of sampled characteristics used during gradient descent.
If identifiability fails, as in overparameterized or rank-deficient settings, the asymptotic
parameter rate need not be reached, yet the iteration still drives the weak loss to a low residual, so accuracy persists even if parameter recovery is impossible.}
\qw{The stochastic gradient iteration is assessed by two criteria that, in an inverse problem, need not agree.
\emph{Accuracy} is the extent to which the recovered dynamics reproduce the observed marginal distributions, measured by the decay of the weak loss.
\emph{Stability} is the extent to which the true parameters $\boldsymbol{\alpha}^\star$ are themselves recovered.
The two must be distinguished because the data are marginal, and different parameters can produce nearly identical marginal statistics.
Accuracy is therefore attainable if the iteration converges, whereas stability holds only if the observed statistics determine $\boldsymbol{\alpha}^\star$ uniquely.
If identifiability fails, as in overparameterized or rank-deficient settings, the asymptotic parameter rate need not be reached, yet the iteration still drives the loss low, so accuracy persists even if parameter recovery is impossible.}


\subsection{Convergence Rate}

Let the uniform constants attached to Assumption~\ref{ass:char_regularity} be
\begin{equation}
\label{eq:uniform_constants}
\begin{aligned}
L_R &:= \sup_{(\boldsymbol{q},t,\boldsymbol{\alpha})}\|\nabla_{\boldsymbol{q}}\boldsymbol{R}(\boldsymbol{q},t;\boldsymbol{\alpha})\|, &
M_R &:= \sup_{(\boldsymbol{q},t,\boldsymbol{\alpha})}\|\partial_{\boldsymbol{\alpha}}\boldsymbol{R}(\boldsymbol{q},t;\boldsymbol{\alpha})\|, \\
L_h &:= \sup_{(\boldsymbol{q},t,\boldsymbol{\alpha})}\|\nabla_{\boldsymbol{q}}h_{\boldsymbol{\alpha}}(\boldsymbol{q},t)\|, &
M_h &:= \sup_{(\boldsymbol{q},t,\boldsymbol{\alpha})}\|\partial_{\boldsymbol{\alpha}}h_{\boldsymbol{\alpha}}(\boldsymbol{q},t)\|,
\end{aligned}
\end{equation}
together with the finite test-function bound $M_\varphi:=\max_{1\le k\le K}\|\varphi_k\|_{L^\infty(\mathbb{R}^{d_x})}$ and the finite score second moment $C_{\nabla\! f}^2:=\mathbb{E}_{\boldsymbol{q}_0\sim f_0}\|\nabla\log f_0(\boldsymbol{q}_0)\|^2$.

\begin{lemma}[Pathwise bounds and the Eulerian second moment]
\label{lem:pathwise-bounds}
\qw{
Let Assumption~\ref{ass:char_regularity} and the constants~\eqref{eq:uniform_constants} hold.
The backward sensitivity $S_{\boldsymbol{\alpha}}$ of Theorem~\ref{thm:G_formula} obeys $\|S_{\boldsymbol{\alpha}}(t)\|\le\Sigma_{T-t}$ for every $t\in[0,T]$, where
\begin{equation}
\label{eq:S-bound}
\Sigma_s
:=
M_R\,\frac{e^{L_R s}-1}{L_R},
\end{equation}
with the convention $\Sigma_s=M_R s$ at $L_R=0$; the bound is largest at $t=0$, so it holds uniformly on $[0,T]$ with the single constant $\Sigma_T$.
Consequently the parametric divergence sensitivity obeys $\|r(T)\|\le C_r:=(M_h+L_h\Sigma_T)T$, and the Eulerian sensitivity has second moment
\begin{equation}
\label{eq:G-second-moment}
\mathbb{E}\|G\|^2
\le
\Gamma_T^2
:=
2C_r^2+2C_{\nabla\! f}^2\Sigma_T^2 ,
\end{equation}
the expectation being taken over $\boldsymbol{q}_0\sim f_0$.
}
\end{lemma}

\qw{
\begin{proof}
Integrating~\eqref{eq:backward-sensitivity} backward from $S_{\boldsymbol{\alpha}}(T)=\boldsymbol{0}$ gives
\[
S_{\boldsymbol{\alpha}}(t)
=
-\int_t^T\Psi(t,s)\,\partial_{\boldsymbol{\alpha}}\boldsymbol{R}(\boldsymbol{q}(s),s;\boldsymbol{\alpha})\,ds,
\]
with $\Psi$ the state-transition matrix of $\dot{\boldsymbol{q}^{\prime}}=\nabla_{\boldsymbol{q}}\boldsymbol{R}\,\boldsymbol{q}^{\prime}$.
Gr\"onwall's inequality with $\|\nabla_{\boldsymbol{q}}\boldsymbol{R}\|\le L_R$ gives $\|\Psi(t,s)\|\le e^{L_R(s-t)}$ for $s\ge t$, and $\|\partial_{\boldsymbol{\alpha}}\boldsymbol{R}\|\le M_R$ then yields
\[
\|S_{\boldsymbol{\alpha}}(t)\|
\le
M_R\int_t^Te^{L_R(s-t)}\,ds
=
\Sigma_{T-t},
\]
which increases as $t$ decreases.
Taking norms in~\eqref{eqn:r-formula} and using $\|\partial_{\boldsymbol{\alpha}}h_{\boldsymbol{\alpha}}\|\le M_h$ with $\|(\nabla_{\boldsymbol{q}}h_{\boldsymbol{\alpha}})^{\!\top}S_{\boldsymbol{\alpha}}(t)\|\le L_h\Sigma_T$, which requires the bound uniformly on $[0,T]$ and not only at $t=0$, gives $\|r(T)\|\le C_r$.
Finally the decomposition~\eqref{eqn:G-decomposition} and Young's inequality give $\|G\|^2\le2\|r(T)\|^2+2\|S_{\boldsymbol{\alpha}}(0)\|^2\|\nabla\log f_0(\boldsymbol{q}(0))\|^2$; the backward characteristic through $\boldsymbol{q}_T=\Phi_T(\boldsymbol{q}_0;\boldsymbol{\alpha})$ returns to $\boldsymbol{q}(0)=\boldsymbol{q}_0$, so the expectation of the last factor against $f_0$ is $C_{\nabla\! f}^2$, which gives~\eqref{eq:G-second-moment}.
\end{proof}

The amplification $e^{L_RT}$ carried by $\Sigma_T$ is the backward-anchored counterpart of the forward bound $\|\Psi(T,s)\|\le e^{L_R(T-s)}$ on the sensitivity anchored at $\boldsymbol{q}_0$.
The same tangent growth that degrades the conditioning of $\mathcal{I}_T$ reappears here as growth of the gradient noise, so signal and noise share the rate and the usable time horizon is capped rather than merely expensive.

Each iteration evaluates the crossed estimator~\eqref{eq:crossed-estimator} on a fresh ensemble of $N$ characteristics, then forms the preconditioner~\eqref{eq:shampoo-update} and takes the parameter step~\eqref{eq:sgd_update} along it.
Both updates see the true gradient only through $\widehat{\boldsymbol{g}}$, so its deviation from that gradient is what the iteration inherits, and the size of the deviation in $N$ sets the rate of convergence.

\begin{theorem}[Finite-sample variance of the crossed gradient estimator]
\label{thm:lip-variance}
Let $\{\boldsymbol{q}_0^{(n)}\}_{n=1}^N\sim f_0$ be i.i.d.\ with $N\ge2$ and let the hypotheses of Lemma~\ref{lem:pathwise-bounds} hold.
For a fixed feature index $k$ write $\mu_k^{(n)}:=\varphi_k(\Pi_{\boldsymbol{X}}\boldsymbol{q}_T^{(n)})$ and
\[
\Delta_k:=\mu_k(\boldsymbol{\alpha})-\tilde{\mu}_k,
\qquad
b_k:=\mathbb{E}\bigl[\mu_k^{(n)}G^{(n)}\bigr]=\partial_{\boldsymbol{\alpha}}\mu_k(\boldsymbol{\alpha}),
\]
the second identity by Lemma~\ref{lem:score_identity}, together with the per-sample spreads $\sigma_k^2:=\mathrm{Var}\bigl(\mu_k^{(n)}\bigr)$ and $v_k^2:=\mathbb{E}\|\mu_k^{(n)}G^{(n)}-b_k\|^2$.
The per-feature crossed U-statistic estimator,
\begin{equation}
\label{eq:gk-hat}
\widehat g_k
:=
\frac{1}{N(N-1)}
\sum_{i\neq j}
\bigl(\mu_k^{(i)}-\tilde{\mu}_k\bigr)\,\mu_k^{(j)}\,G^{(j)}
\in\mathbb{R}^p ,
\end{equation}
then satisfies, for every $N\ge2$,
\begin{equation}
\label{eq:var-gk-bound}
\Bigl(\mathbb{E}\bigl\|\widehat g_k-\mathbb{E}[\widehat g_k]\bigr\|^2\Bigr)^{1/2}
\le
\frac{\sigma_k\|b_k\|+|\Delta_k|\,v_k}{\sqrt{N}}
+
\frac{2M_\varphi v_k}{\sqrt{N}-1} ,
\end{equation}
and in particular $\mathbb{E}\|\widehat g_k-\mathbb{E}[\widehat g_k]\|^2\le49\,M_\varphi^4\Gamma_T^2/N$ for $N\ge4$, with $\Gamma_T$ given by~\eqref{eq:G-second-moment}.
The variance is therefore $O(1/N)$ with a constant that grows as $e^{2L_RT}$ through $\Sigma_T$.
The first contribution carries the data misfit $\Delta_k$ as a factor and vanishes as the fit becomes exact, so near $\boldsymbol{\alpha}^\star$ the gradient noise is governed by $\sigma_k\|b_k\|$ and by the higher-order term.
\end{theorem}

\begin{proof}
Write $A^{(n)}:=\mu_k^{(n)}-\tilde{\mu}_k=\Delta_k+a^{(n)}$ and $B^{(n)}:=\mu_k^{(n)}G^{(n)}=b_k+\beta^{(n)}$ with $a^{(n)},\beta^{(n)}$ centered, $|a^{(n)}|\le2M_\varphi$ almost surely, $\mathbb{E}[(a^{(n)})^2]=\sigma_k^2$ and $\mathbb{E}\|\beta^{(n)}\|^2=v_k^2$.
Separating the double sum into its full and diagonal parts and substituting these decompositions gives
\[
\widehat g_k-\Delta_k b_k
=
\bar a\,b_k+\Delta_k\bar\beta
+
\frac{N}{N-1}\,\bar a\bar\beta
-
\frac{1}{N(N-1)}\sum_{n=1}^N a^{(n)}\beta^{(n)},
\]
with $\bar a$ and $\bar\beta$ the sample means, the cancellation of the $\Delta_kb_k$, $\bar ab_k$ and $\Delta_k\bar\beta$ coefficients being the unbiasedness of Proposition~\ref{prop:crossed_ustatistic}.
The first two terms form an average of $N$ i.i.d.\ centered vectors, so their $L^2$ norm is $N^{-1/2}\bigl(\mathbb{E}\|a b_k+\Delta_k\beta\|^2\bigr)^{1/2}\le N^{-1/2}(\sigma_k\|b_k\|+|\Delta_k|v_k)$ by the Minkowski inequality.
For the remaining two terms, $|\bar a|\le2M_\varphi$ almost surely and $\|\bar\beta\|_{L^2}=v_k/\sqrt N$ give $\|\bar a\bar\beta\|_{L^2}\le2M_\varphi v_k/\sqrt N$, while the triangle inequality gives $\|\sum_n a^{(n)}\beta^{(n)}\|_{L^2}\le2NM_\varphi v_k$, so that together they are bounded in $L^2$ by
\[
2M_\varphi v_k\Bigl(\frac{\sqrt N}{N-1}+\frac{1}{N-1}\Bigr)
=
\frac{2M_\varphi v_k}{\sqrt N-1} .
\]
One more application of the Minkowski inequality gives~\eqref{eq:var-gk-bound}.
Boundedness of the test functions gives $\sigma_k\le M_\varphi$, $|\Delta_k|\le2M_\varphi$, $\|b_k\|\le M_\varphi\Gamma_T$ and $v_k\le M_\varphi\Gamma_T$, and $(\sqrt N-1)^{-1}\le2N^{-1/2}$ for $N\ge4$, whence the right-hand side is at most $7M_\varphi^2\Gamma_TN^{-1/2}$.
\end{proof}
}

\begin{remark}[Computational consequences]
\label{rem:practical}
The variance constant in Theorem~\ref{thm:lip-variance} grows rapidly with the product $L_R T$, reflecting the amplification of sensitivity along unstable characteristics.  
Consequently, the time horizon should remain moderate in strongly chaotic regimes.  
\end{remark}

\begin{remark}[Polynomial test functions]
\qw{The explicit variance constant in Theorem~\ref{thm:lip-variance} applies only to bounded test functions. For the unbounded polynomial moments used below, unbiasedness and the $O(N^{-1})$ variance rate instead require the corresponding second moments to be finite; the displayed constant is not claimed.}
\end{remark}

\begin{remark}[Stochastic-approximation bound]
\label{rem:sgd-rate}
Combining Theorem~\ref{thm:lip-variance} with the 
nonconvex SGD descent lemma~\cite{bottou2018optimization}, suppose that $\partial_{\boldsymbol{\alpha}}\mathcal{J}_{K,M}$ is $L_g$-Lipschitz on the relevant sublevel set, the stochastic gradient is unbiased with uniformly bounded second moment, and the preconditioners are uniformly positive definite and bounded:
\[
0<\lambda_- I_p \preceq H_\ell^{-1/2}
\preceq \lambda_+ I_p < \infty.
\]

For the step-size schedule $\eta_\ell = \frac{\eta_0}{1+\ell/\tau}$, application of the 
nonconvex SGD descent inequality in \cite{bottou2018optimization} leads to
\begin{equation}
\label{eq:stoch-appr-bound}
\min_{1\le\ell\le L}
\mathbb{E}
\|
\partial_{\boldsymbol{\alpha}}
\mathcal{J}_{K,M}(\boldsymbol{\alpha}_\ell)
\|^2
=
O\!\left(\frac{1}{\log L}\right).
\end{equation}
up to constants depending on the variance bound~\eqref{eq:G-second-moment}, $\lambda_\pm$, and the initial objective gap.

\end{remark}
\subsection{Stability and accuracy}
\label{subsec:stability_accuracy}

The stochastic approximation bound in~\eqref{eq:stoch-appr-bound} controls the expected norm of the gradient of the weak loss rather than the parameter estimation error itself. Consequently, although the weak loss improves the stability of the iterative parameter inference, it does not necessarily guarantee accurate recovery of the parameter or the associated probability density function. As is common in inverse problems, this reflects a trade-off between stability and accuracy.

\paragraph{Stability regime}
In the stable regime, the convergence rates of the parameter and the weak loss are different, and a comparison of these rates can thus be used to identify the stability regime. Both convergence rates can be estimated starting from the assumption that the Fisher information matrix $ \mathcal{I}_T=\mathcal{S}^{\top}\mathcal{S}$ is positive definite at the reference parameter $\boldsymbol{\alpha}^\star$.  
The convergence rate of the local parameter error, $\boldsymbol{e}_\ell
:=
\boldsymbol{\alpha}_\ell-\boldsymbol{\alpha}^\star.$
can be estimated from a
local linearization of the stochastic-gradient iteration using stochastic-approximation theory~\cite{fabian1968asymptotic,borkar2008stochastic}
\[
\boldsymbol{e}_{\ell+1}
=
\left(
I-\eta_\ell P\,\mathcal{I}_T
\right)\boldsymbol{e}_\ell
-
\eta_\ell P\,\boldsymbol{\xi}_\ell,
\]
where $P\succ0$ is the limiting preconditioner and $\boldsymbol{\xi}_\ell$ is the stochastic-gradient noise. This noise is zero-centered because the crossed U-statistic estimator is unbiased, and its covariance is of order $N^{-1}$ by the finite-sample variance bound~\eqref{eq:var-gk-bound}, which is where the $N^{-1}$ in~\eqref{eq:msa_rate} originates.

For the inverse-time schedule $\eta_\ell \sim \frac{a}{\ell}$,
stochastic-approximation asymptotics gives  
the mean-square parameter error satisfies
\begin{equation}
\label{eq:msa_rate}
\mathbb{E}\|\boldsymbol{\alpha}_\ell-\boldsymbol{\alpha}^\star\|^2
=
O\!\left(\frac{1}{N\ell}\right),
\end{equation}
provided that
\begin{equation}
    \label{eqn:strong_requirement}
a\,\lambda_{\min}(P\mathcal{I}_T)>\frac12,
\end{equation}
And the local quadratic model gives the expected weak loss with the same asymptotic scaling:
\[
\mathbb{E}\bigl[
\mathcal{J}_{K,M}(\boldsymbol{\alpha}_\ell)
\bigr] - \mathcal{J}_{K,M}(\boldsymbol{\alpha}^\star)
=
O\!\left(\frac{1}{N\ell}\right).
\]
Thus, in the identifiable regime, the parameter error decays as $\ell^{-1/2}$ while the weak loss decays as $\ell^{-1}$. In many modern overparameterized machine-learning settings, however, the sensitivity matrix is highly rank-deficient or poorly conditioned, so the condition~\eqref{eqn:strong_requirement} may fail and the identifiable asymptotic regime need never be attained. Convergence toward low weak-loss residuals remains meaningful even if parameter recovery itself is ill-posed.

The convergence constants depend on the conditioning of $\mathcal{I}_T$.  
For the unpreconditioned iteration $P=I_p$, small values of $\lambda_{\min}(\mathcal{I}_T)$ lead to slow convergence and poor parameter stability.  
A full-matrix adaptive preconditioner can partially mitigate this anisotropy by approximately whitening dominant gradient-covariance directions.  
However, adaptive methods such as Shampoo track the stochastic gradient second moment $\mathbb{E}[\widehat{\boldsymbol g}_\ell\widehat{\boldsymbol g}_\ell^\top]$
rather than the sensitivity Gram matrix $\mathcal{I}_T$ itself, so alignment between the two generally requires additional assumptions on the local gradient-noise covariance.

The local quadratic model also yields the residual-to-error estimate
\begin{equation}
\label{eq:residual_error}
\|\boldsymbol{\alpha}-\boldsymbol{\alpha}^\star\|
\le
\sqrt{
\frac{
2\mathcal{J}_{K,M}(\boldsymbol{\alpha})
}{
\lambda_{\min}(\mathcal{I}_T)
}
},
\end{equation}
which shows that a small weak loss certifies a small parameter error only if the inverse problem is locally well-conditioned.

\paragraph{Accuracy regime}
The situation changes fundamentally if the parameter-to-feature map is non-injective.  
If $p \gg MK $, the sensitivity matrix $\mathcal{S}$ is necessarily rank deficient, and many distinct parameter vectors generate indistinguishable collections of weak features.  
In this regime, convergence of the iterates $\boldsymbol{\alpha}_\ell$ toward a unique reference parameter is neither expected nor meaningful.

The relevant criterion is instead the decay of the weak loss, $\mathcal{J}_{K,M}(\boldsymbol{\alpha}_\ell)$, 
or equivalently, the agreement between predicted and observed marginal statistics.  
The stochastic-approximation estimate of Remark~\ref{rem:sgd-rate} gives
\[
\min_{1\le\ell\le L}
\mathbb{E}
\|
\partial_{\boldsymbol{\alpha}}
\mathcal{J}_{K,M}(\boldsymbol{\alpha}_\ell)
\|^2
=
O\!\left(\frac{1}{\log L}\right),
\]
independently of the parameter dimension $p$ and without requiring invertibility of $\mathcal{I}_T$.
Thus, even in non-identifiable settings, stochastic optimization still drives the weak residual toward stationarity.

\begin{table}[h]
\centering
\begin{tabular}{llccc}
\hline
& metric  & $N$ & $\ell$ & $T$ \\
\hline
gradient noise
& $\bigl(\mathbb{E}\|\widehat{\boldsymbol{g}}-\mathbb{E}[\widehat{\boldsymbol{g}}]\|^2\bigr)^{1/2}$
& $N^{-1/2}$ & --- & $e^{L_RT}$ \\
accuracy
& $\mathbb{E}\bigl[\mathcal{J}_{K,M}(\boldsymbol{\alpha}_\ell)\bigr]-\mathcal{J}_{K,M}(\boldsymbol{\alpha}^\star)$
& $N^{-1}$ & $\ell^{-1}$ & $e^{2L_RT}$ \\
stability
& $\bigl(\mathbb{E}\|\boldsymbol{\alpha}_\ell-\boldsymbol{\alpha}^\star\|^2\bigr)^{1/2}$
& $N^{-1/2}$ & $\ell^{-1/2}$ & $e^{L_RT}$ \\
stationarity
& $\min_{1\le\ell\le L}\mathbb{E}\|\partial_{\boldsymbol{\alpha}}\mathcal{J}_{K,M}(\boldsymbol{\alpha}_\ell)\|^2$
& --- & $1/\log L$ & $e^{2L_RT}$ \\
\hline
\end{tabular}
\vspace{5pt}
\caption{Dependence of the gradient noise, the weak loss, the parameter error, and the weak-loss gradient on the size of characteristic ensemble $N$, the iteration count $\ell$, and the time horizon $T$. The $N$ and $\ell$ entries are decay rates; the $T$ entries are the growth of the constant with observation time.}
\label{tab:rates}
\end{table}

To summarize, Table~\ref{tab:rates} collates the metrics for convergence.
The number of ensemble members, $N$, acts on the rates only through the gradient noise, so raising $N$ lowers the curves without changing their slopes, whereas the iteration count carries the rates themselves.
The horizon enters the constants rather than the rates, through $\Gamma_T$, and at the same exponential rate as the sensitivity, so the two offset each other and the usable horizon is bounded.
Accuracy and stationarity need no invertibility, whereas stability rests on $\mathcal{I}_T\succ\boldsymbol{0}$ and on the step-size condition~\eqref{eqn:strong_requirement}, and it is exactly this pair of requirements that fails in the overparameterized regime.

\section{Numerical experiments}
\label{sec:experiments}


\subsection{Experiment 1: Hidden modes of a linear system from a single observable}
\label{subsec:oscillator}

\subsubsection{Model description}

The first experiment is a minimal, well-conditioned instance of the inverse
problem with genuine marginalization: three coupled linear modes are evolved,
but only the first is observed, and the two hidden modes must be inferred from
the observable's marginal alone.
Every quantity \del{of \S\ref{sec:problem_setup}--\S\ref{sec:algorithm}}of the preceding analysis is available
in closed form, so the asymptotic rate~\eqref{eq:msa_rate} can be validated
without the confounding effect of ill-conditioning.
Consider the three-dimensional damped rotation chain
\begin{equation}
\label{eq:osc_ode}
\dot{\boldsymbol{q}}(t) = A(\boldsymbol{\alpha})\,\boldsymbol{q}(t),
\qquad
A(\boldsymbol{\alpha}) =
\begin{pmatrix} -\gamma & \omega_1 & 0 \\ -\omega_1 & -\gamma & \omega_2 \\ 0 & -\omega_2 & -\gamma \end{pmatrix},
\qquad
\boldsymbol{q}(0)\sim\mathcal{N}(\boldsymbol{\mu}_0,\Sigma_0),
\end{equation}
with three unknown parameters $\boldsymbol{\alpha}=(\gamma,\omega_1,\omega_2)$: a common damping
$\gamma$ and two skew couplings $\omega_1,\omega_2$ between neighboring modes.
In the notation of the general formulation~\eqref{eqn:ODE}, the state is
$\boldsymbol{q}=(x,y_1,y_2)\in\mathbb{R}^3$ and the vector field is the linear,
autonomous map $\boldsymbol{R}(\boldsymbol{q},t;\boldsymbol{\alpha})=A(\boldsymbol{\alpha})\,\boldsymbol{q}$.
Only the first mode $x$ is observed; $y_1$ and $y_2$ are hidden and enter the
inference solely through their imprint on the marginal of $x$.
The observable is a damped multi-frequency oscillation whose envelope encodes
the damping $\gamma$ and whose spectrum encodes the couplings $\omega_1,\omega_2$, so all three
parameters are strongly imprinted on the marginal of $x$.
Both the marginal moments and the sensitivity Jacobian~\eqref{eq:sens_jacobian}
follow analytically from $e^{A(\boldsymbol{\alpha})t}$, which makes this system an
exact reference for the Monte Carlo estimators and the convergence rate.

\subsubsection{Experimental setup}

The true parameters are $\boldsymbol{\alpha}^\star=(\gamma^\star,\omega_1^\star,\omega_2^\star)=(0.5,2,1)$,
the initial distribution is $\boldsymbol{q}_0\sim\mathcal{N}(\boldsymbol{\mu}_0,\Sigma_0)$
with $\boldsymbol{\mu}_0=(1,0.5,0.3)^{\!\top}$ and $\Sigma_0=0.25\,I_3$, and the
observation times are $M=5$ points equally spaced in $[0.3,3]$.
At each time only the first two raw moments $\mathbb{E}[x]$ and $\mathbb{E}[x^2]$
of the observed mode are recorded ($K=2$ test functions), giving $MK=10$ scalar
constraints for $p=3$ parameters.
At $\boldsymbol{\alpha}^\star$ the sensitivity Jacobian has singular values
$\{1.300,0.837,0.523\}$, so the Jacobian condition number is
$\kappa(\mathcal{S})\approx2.48$ and the sensitivity Gramian
$\mathcal{I}_T=\mathcal{S}^{\!\top}\mathcal{S}$ has $\kappa(\mathcal{I}_T)\approx6.17$.
The solution is unique: the sensitivity Gramian is positive definite, so
$\boldsymbol{\alpha}^\star$ is a locally isolated minimizer, and the nonzero hidden
means $\mu_{0,1},\mu_{0,2}$ break the $(\omega_1,\omega_2)\mapsto(-\omega_1,-\omega_2)$ reflection
symmetry of the couplings, so no sign ambiguity remains; a scan of the parameter
box recovers $\boldsymbol{\alpha}^\star$ as the only match to the observed moments.
The problem is thus well conditioned, and the local strong-convexity hypothesis \del{of
\S\ref{subsec:stability_accuracy} }holds on a large neighborhood of
$\boldsymbol{\alpha}^\star$, so a single initialization suffices to exhibit the rate.

\subsubsection{Results}
\label{sec:results_linear}
At every iteration of Algorithm~\ref{alg:crossed-sgd}, a fresh set of characteristics
$\{\boldsymbol{q}_0^{(n)}\}_{n=1}^N\sim\mathcal{N}(\boldsymbol{\mu}_0,\Sigma_0)$
drives the crossed U-statistic gradient~\eqref{eq:multi_time_estimator};
the Shampoo preconditioner~\eqref{eq:shampoo-update} uses $\beta=0.99$,
$\varepsilon=10^{-3}$, and the step size follows $\eta_\ell=0.5/(\ell+50)$.

Before running the gradient descent, we verify that the crossed U-statistic gradient
$\widehat{\boldsymbol{g}}$ and the empirical sensitivity Jacobian
$\widehat{\mathcal{S}}$ converge to their closed-form references at the Monte Carlo
rate. At the test point $\boldsymbol{\alpha}_{\rm test}=\boldsymbol{\alpha}^\star+0.05\,(1,1,1)/\sqrt3$
we draw $n_{\rm trials}=400$ independent sets of characteristics at each
$N\in\{16,64,256,1024,4096,16384\}$.
Figure~\ref{fig:osc_validation} reports the four diagnostics.
The standard deviation of $\widehat{g}_j$ decays as $N^{-1/2}$ across three
decades of $N$, the rate
predicted by Theorem~\ref{thm:lip-variance}. The bias $|\mathbb{E}[\widehat{g}_j]-g_j|$
remains within about one Monte Carlo standard error of the mean at every $N$,
consistent with the exact unbiasedness of \del{Proposition~\ref{prop:crossed_ustatistic}}the crossed U-statistic estimator.
At $N=4096$ the per-component mean of $\widehat{g}_j$ matches the exact gradient within
the 95\% confidence interval, and the empirical Jacobian $\widehat{\mathcal{S}}$
converges to $\mathcal{S}$ at the same $N^{-1/2}$ rate.

\begin{figure}[h!]
    \centering
    \includegraphics[width=\linewidth]{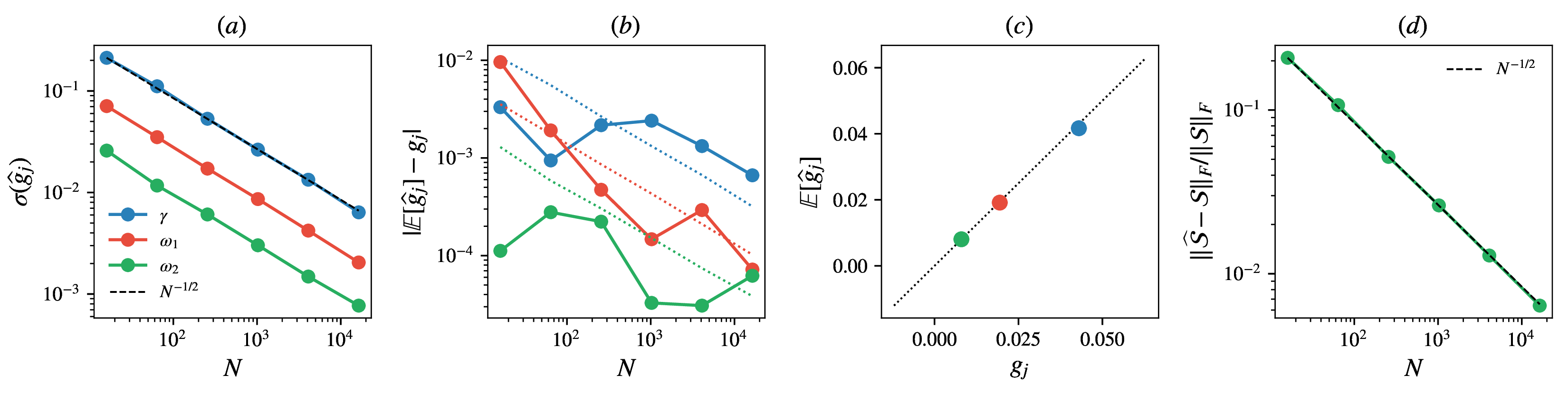}
    \vspace{-20pt}
    \caption{Validation of $\widehat{\boldsymbol{g}}$ and $\widehat{\mathcal{S}}$ at
    $\boldsymbol{\alpha}_{\rm test}$ with $\|\boldsymbol{\alpha}_{\rm test}-\boldsymbol{\alpha}^\star\|=0.05$,
    using $n_{\rm trials}=400$ independent sets of characteristics at each $N$.
    $(a)$~Standard deviation of $\widehat{g}_j$ versus $N$, against the $N^{-1/2}$ guide line.
    $(b)$~Bias $|\mathbb{E}[\widehat{g}_j]-g_j|$ versus $N$ (solid), against the Monte Carlo
    standard error $\sigma_j/\sqrt{n_{\rm trials}}$ (dotted).
    $(c)$~Mean of $\widehat{g}_j$ at $N=4096$ versus the exact $g_j$, with 95\% confidence
    interval; dotted line is $y=x$.
    $(d)$~Relative Frobenius distance $\|\widehat{\mathcal{S}}-\mathcal{S}\|_F/\|\mathcal{S}\|_F$
    versus $N$, against the $N^{-1/2}$ guide line.}
    \label{fig:osc_validation}
\end{figure}

Figure~\ref{fig:osc_convergence} reports the convergence of Algorithm~\ref{alg:crossed-sgd}
from the initialization $\boldsymbol{\alpha}_0=(0.8,1.5,1.5)$, at distance
$\|\boldsymbol{\alpha}_0-\boldsymbol{\alpha}^\star\|=0.61$ from the truth, over
$4\times10^4$ iterations and $N\in\{16,64,256,1024\}$, averaged over sixteen seeds.
Because the problem is well conditioned, both halves of the rate
prediction~\eqref{eq:msa_rate} hold globally, with no basin restriction.
The weak loss $\mathcal{J}_{K,M}(\boldsymbol{\alpha}_\ell)$ decays along the predicted
$\ell^{-1}$ rate.
The parameter error $\|\boldsymbol{\alpha}_\ell-\boldsymbol{\alpha}^\star\|$ decays along
the predicted $\ell^{-1/2}$ rate and a wider range of
iterations enters the asymptotic regime.
The curves also stack vertically by $N$, the error level dropping as the $1/N$ gradient
variance of Theorem~\ref{thm:lip-variance} would dictate; the corresponding $N^{-1/2}$
scaling of the gradient estimator itself is shown directly in
Figure~\ref{fig:osc_validation}$(a)$.
This experiment validates both halves of the rate prediction cleanly on a system where
the sensitivity Gramian is well conditioned, isolating the asymptotic rate from any
conditioning-limited transient.

\begin{figure}[h!]
    \centering
    \includegraphics[width=\linewidth]{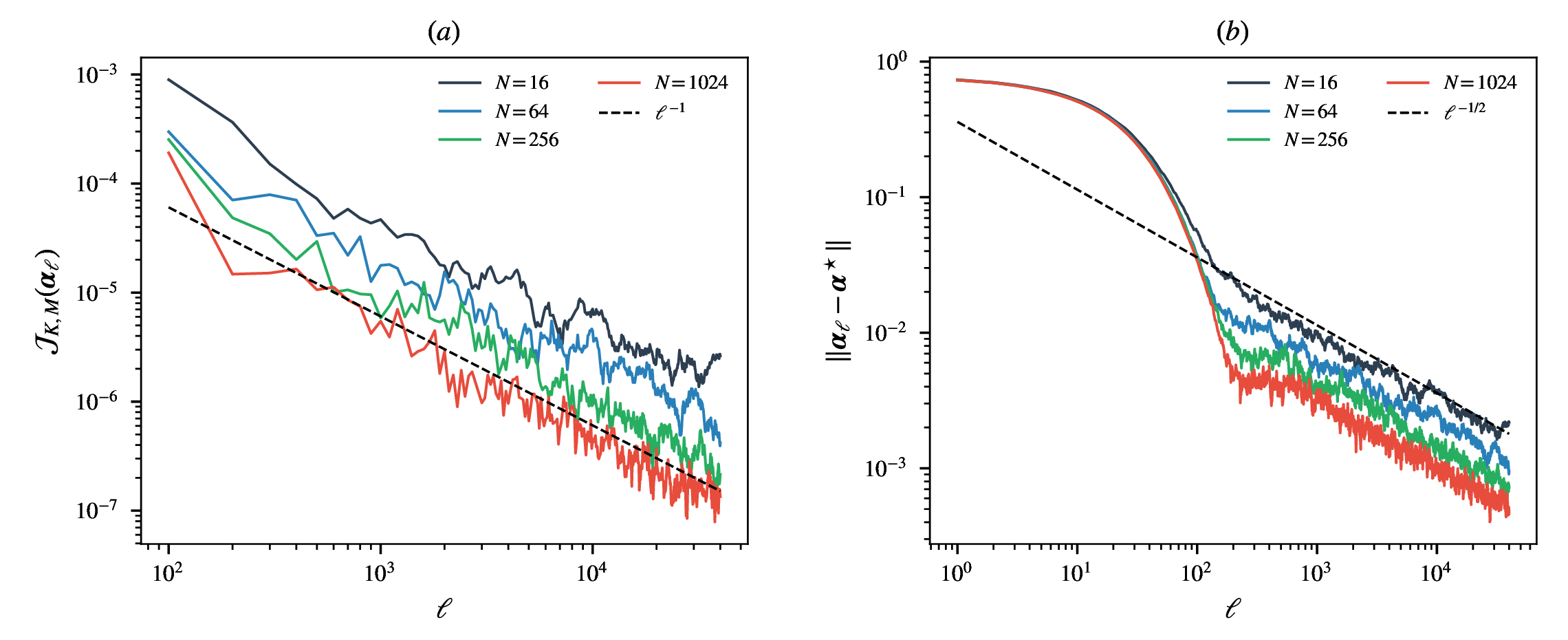}
    \vspace{-8pt}
    \caption{Convergence of Algorithm~\ref{alg:crossed-sgd} on the three-dimensional
    damped rotation chain with $\boldsymbol{\alpha}^\star=(0.5,2,1)$, step
    $\eta_\ell=0.5/(\ell+50)$,
    $4\times10^4$ iterations, $N\in\{16,64,256,1024\}$, averaged over sixteen seeds.
    $(a)$~Weak loss $\mathcal{J}_{K,M}(\boldsymbol{\alpha}_\ell)$ versus $\ell$, against
    the $\ell^{-1}$ guide line.
    $(b)$~Parameter error $\|\boldsymbol{\alpha}_\ell-\boldsymbol{\alpha}^\star\|$ versus
    $\ell$, against the $\ell^{-1/2}$ guide line.}
    \label{fig:osc_convergence}
\end{figure}

\subsection{Experiment 2: Nonlinear growth with a hidden mode}
\label{subsec:gompertz}

\subsubsection{Model description}

The second experiment carries the exact convergence check of \S\ref{subsec:oscillator}
to a genuinely nonlinear vector field.
The observable $w>0$ obeys a Gompertz growth law modulated by a hidden decaying
mode $\theta$,
\begin{equation}
\label{eq:gomp_ode}
\dot{w} = -a\,w\,\ln(w/w_\star) + c\,w\,\theta,
\qquad
\dot{\theta} = -b\,\theta,
\qquad
w_\star = e^{u_\star},
\end{equation}
so that in the notation of~\eqref{eqn:ODE} the state is $\boldsymbol{q}=(w,\theta)$ and the
vector field $\boldsymbol{R}(\boldsymbol{q};\boldsymbol{\alpha})=(-a\,w\ln(w/w_\star)+c\,w\theta,\,-b\,\theta)$
is nonlinear in $w$ through the $w\ln w$ and $w\theta$ terms.
Only $w$ is observed; $\theta$ is hidden.
The nonlinearity is exactly solvable: in the log variable $u=\ln w$ the system is
linear,
\begin{equation}
\label{eq:gomp_log}
\dot{u} = -a\,(u-u_\star) + c\,\theta,
\qquad
\dot{\theta} = -b\,\theta,
\end{equation}
so with a Gaussian initial law $w$ remains lognormal for all $t$ and every moment is
available in closed form,
\begin{equation}
\label{eq:gomp_moments}
\mathbb{E}[w^k(t)] = \exp\!\bigl(k\,\mu_u(t) + \tfrac12 k^2\,\sigma_u^2(t)\bigr),
\end{equation}
with $\mu_u(t),\sigma_u^2(t)$ the mean and variance of $u(t)$ from the linear
system~\eqref{eq:gomp_log}.
The experiment therefore has a nonlinear flow yet exact, known-parameter reference
moments, giving a solid convergence check on a nonlinear model before the black-box
neural-network vector fields introduced later.

\subsubsection{Experimental setup}

The inferred parameters are the Gompertz rate and the hidden coupling
$\boldsymbol{\alpha}=(a,c)$, with true values $\boldsymbol{\alpha}^\star=(1,0.8)$;
the hidden decay $b=1.5$ and the setpoint $u_\star=0$ are known.
The initial distribution is $(u_0,\theta_0)\sim\mathcal{N}(\boldsymbol{\mu}_0,\Sigma_0)$
with $\boldsymbol{\mu}_0=(0.2,1)^{\!\top}$ and $\Sigma_0=\mathrm{diag}(0.09,0.25)$,
and the first two moments $\mathbb{E}[w],\mathbb{E}[w^2]$ are observed
($K=2$ test functions) at $M=5$ times equally spaced in $[0.3,3]$, giving $MK=10$
constraints for $p=2$ parameters.
At $\boldsymbol{\alpha}^\star$ the sensitivity Jacobian has condition number
$\kappa(\mathcal{S})\approx5.7$; the closed-form moments~\eqref{eq:gomp_moments} agree
with a $4\times10^5$-particle Monte Carlo estimate to within $0.15\%$, confirming the
exact reference.

\subsubsection{Results}
The same numerical setup is adopted as in \S
\ref{sec:results_linear}.
The gradient and Jacobian validation yield similar results as the previous experiment.
Figure~\ref{fig:gomp_convergence} reports convergence from
$\boldsymbol{\alpha}_0=(1.4,1.3)$ over $6\times10^4$ iterations and
$N\in\{16,64,256,1024\}$, averaged over sixteen seeds.
On this nonlinear flow the weak loss decays along the $\ell^{-1}$ rate, the same
rate established analytically for the linear case, with the curves stacking by $N$ as
the $1/N$ gradient variance dictates.
The exact moments~\eqref{eq:gomp_moments} make this a ground-truth check: the recovered
parameters approach the known $\boldsymbol{\alpha}^\star$ at the predicted rate on a
nonlinear model.

\begin{figure}[h!]
    \centering
    \includegraphics[width=\linewidth]{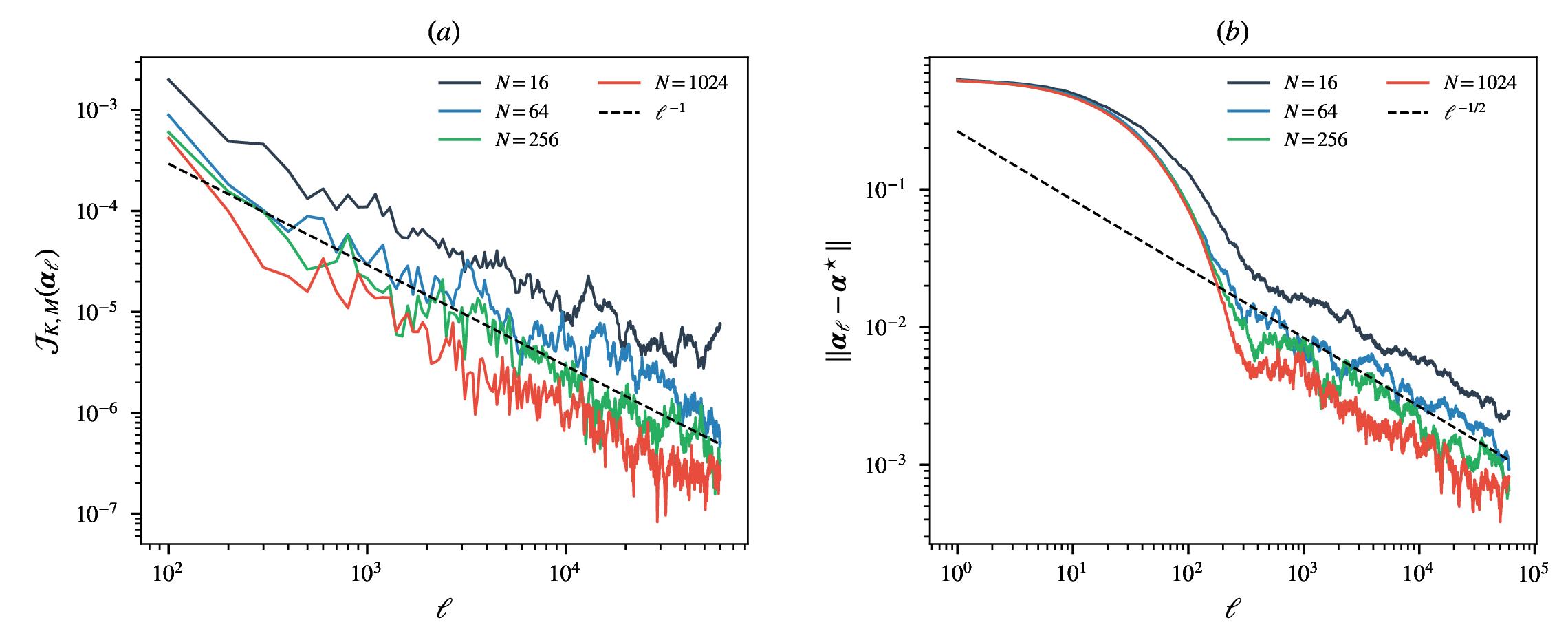}
    \vspace{-8pt}
    \caption{Convergence of Algorithm~\ref{alg:crossed-sgd} on the Gompertz model with
    $\boldsymbol{\alpha}^\star=(1,0.8)$, step $\eta_\ell=0.5/(\ell+50)$, $6\times10^4$
    iterations, $N\in\{16,64,256,1024\}$, sixteen seeds.
    $(a)$~weak loss versus $\ell$ against the $\ell^{-1}$ guide line.
    $(b)$~parameter error versus $\ell$ against the $\ell^{-1/2}$ guide line.}
    \label{fig:gomp_convergence}
\end{figure}

\subsection{Experiment 3: Bimodal marginals from a latent instability}
\label{subsec:bimodal}

\subsubsection{Model description}

The third experiment shows that a hidden variable can create structure in the observed marginal that is difficult for the tested observable-only model to reproduce.
A bistable observable $w$ is tilted by a hidden decaying mode $\theta$,
\begin{equation}
\label{eq:bimodal_ode}
\dot{w} = w\,(a - w^2) + c\,\theta,
\qquad
\dot{\theta} = -\lambda\,\theta,
\end{equation}
so in the notation of~\eqref{eqn:ODE} the state is $\boldsymbol{q}=(w,\theta)$ and the
vector field $\boldsymbol{R}(\boldsymbol{q};\boldsymbol{\alpha})=(w(a-w^2)+c\theta,\,-\lambda\theta)$
is a double well in $w$ with stable points $\pm\sqrt{a}$, perturbed by the latent tilt $c\theta$.
Only $w$ is observed.
Starting from a unimodal initial law for $w$ straddling the unstable point at the
origin, trajectories are driven into the two wells and the observed marginal of $w$
becomes bimodal (Figure~\ref{fig:bimodal_density}).

The latent variable provides a useful augmentation in this experiment.
A one-dimensional autonomous model $\dot{w}=R(w)$ generates its marginals as the
pushforward of the fixed initial law under a single flow, and that flow map
$w_0\mapsto w(t)$ is monotone at every time, because trajectories of a scalar ODE
cannot cross.
Consequently, all marginals must arise as pushforwards of the initial law under the
same one-parameter semigroup of monotone flow maps. This requirement imposes joint
temporal consistency across observation times, rather than allowing each marginal to
be matched independently. As the black-box comparison below shows, the tested scalar
neural ODE is less able than the latent model to match all observed marginals
simultaneously. We emphasize that monotonicity alone does not preclude a change in
modality.
The latent tilt $c\theta$ removes the restriction, since two particles with the same
$w_0$ but different $\theta_0$ are driven to opposite wells, so the induced map on $w$
is not monotone.
Section~\ref{subsub:neural_bimodal} makes this quantitative, fitting both a one-dimensional and a latent black-box vector field to the observed marginals and comparing the marginal sequences they recover.

\begin{figure}[h!]
    \centering
    \includegraphics[width=\linewidth]{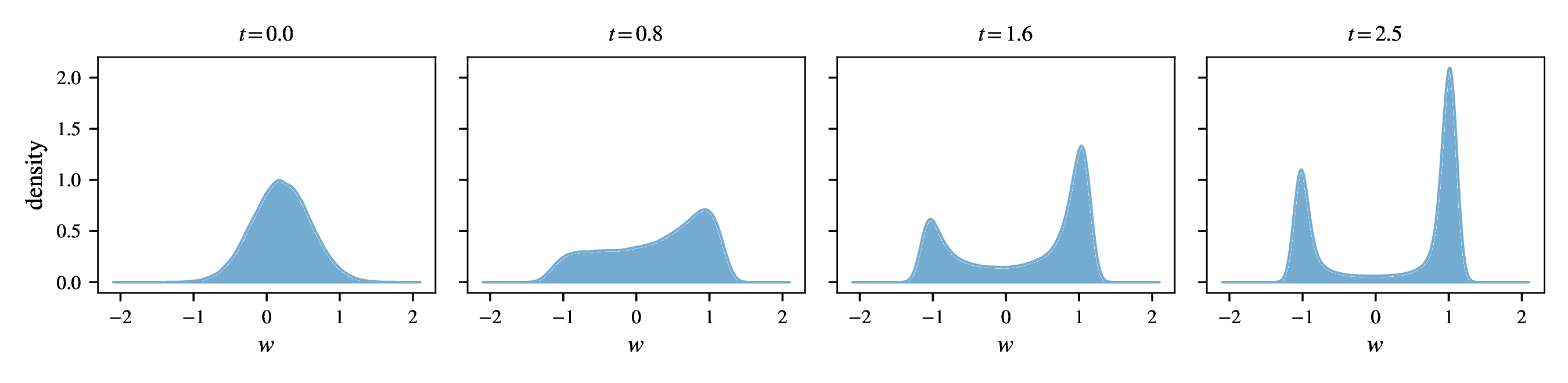}
    \vspace{-8pt}
    \caption{Marginal density of the observed $w$ under~\eqref{eq:bimodal_ode} at
    $\boldsymbol{\alpha}^\star=(1,0.8,1)$, from a unimodal initial law, developing into a
    bimodal distribution with peaks at $w=\pm1$.}
    \label{fig:bimodal_density}
\end{figure}

\subsubsection{Experimental setup}

The parameters are $\boldsymbol{\alpha}=(a,c,\lambda)$, the well location, the latent
coupling, and the latent decay, with true values $\boldsymbol{\alpha}^\star=(1,0.8,1)$.
The initial distribution is $w_0\sim\mathcal{N}(0.2,0.4^2)$ (unimodal) and
$\theta_0\sim\mathcal{N}(0,1)$.
Only $w$ is observed, through the moments $\mathbb{E}[w],\mathbb{E}[w^2],\mathbb{E}[w^4]$
($K=3$ test functions) at $M=5$ times in $[0.5,2.5]$, giving $MK=15$ constraints for $p=3$ parameters.
The flow has no closed form; the per-particle sensitivities are propagated by the
forward variational equations integrated alongside the state, and the reference moments
are computed once from a large fixed ensemble.
At $\boldsymbol{\alpha}^\star$ the sensitivity Jacobian has condition number
$\kappa(\mathcal{S})\approx31$, moderate because the latent decay $\lambda$ is imprinted
on the observed moments only during the transient.

\subsubsection{Parametric recovery}

The gradient and Jacobian validation (Figure~\ref{fig:bimod_validation}) confirms the
Monte Carlo behavior on this nonlinear flow: the gradient standard deviation and the
empirical Jacobian both converge at $N^{-1/2}$ (fitted slopes $-0.50$), with the bias
within the Monte Carlo standard error.
Figure~\ref{fig:bimod_convergence} reports the recovery of $\boldsymbol{\alpha}^\star$ by
Algorithm~\ref{alg:crossed-sgd} from $\boldsymbol{\alpha}_0=(1.3,0.5,1.4)$ over
$2\times10^4$ iterations, eight seeds, at $N\in\{64,256,1024\}$.
The parameter error decays toward the $\ell^{-1/2}$ rate, reaching
$\|\boldsymbol{\alpha}_\ell-\boldsymbol{\alpha}^\star\|\approx3\times10^{-2}$;
the weak-loss floor in figure \ref{fig:bimod_convergence}~$(a)$ is set by the Monte Carlo error of the finite
reference ensemble rather than by the optimizer.
The moderate conditioning lengthens the transient relative to the well-conditioned
Experiments 1 and 2, but Algorithm~\ref{alg:crossed-sgd} recovers all three physical
parameters, including the latent decay $\lambda$, from the observed marginal of $w$
alone.

\begin{figure}[h!]
    \centering
    \includegraphics[width=\linewidth]{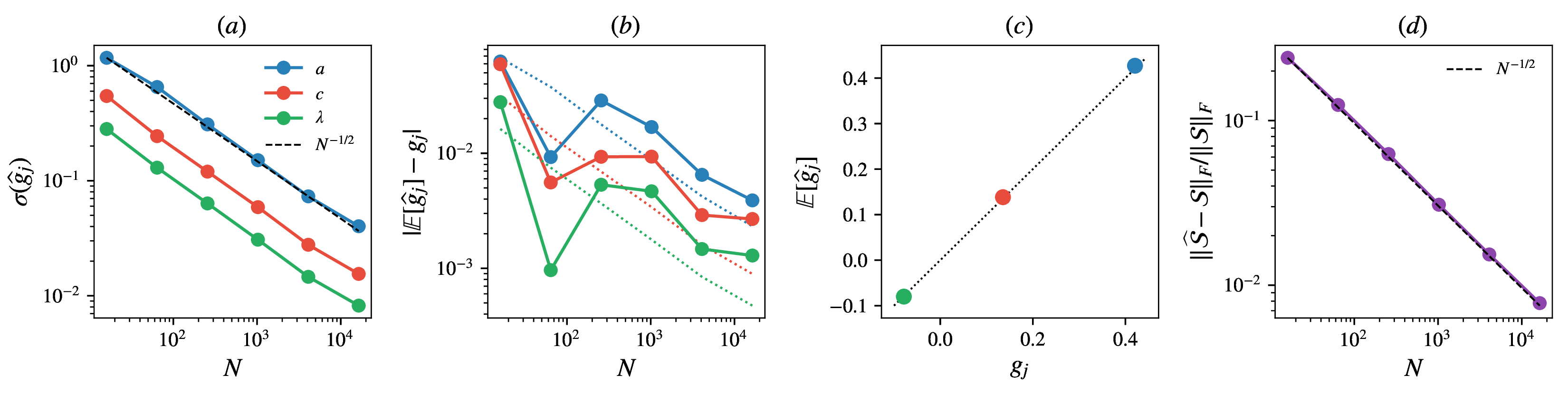}
    \vspace{-20pt}
    \caption{Validation of $\widehat{\boldsymbol{g}}$ and $\widehat{\mathcal{S}}$ for the
    bistable-latent model at $\boldsymbol{\alpha}_{\rm test}$, $n_{\rm trials}=300$
    sets of characteristics per $N$.  Layouts same as in figure~\ref{fig:osc_validation}.}
    \label{fig:bimod_validation}
\end{figure}

\begin{figure}[h!]
    \centering
    \includegraphics[width=\linewidth]{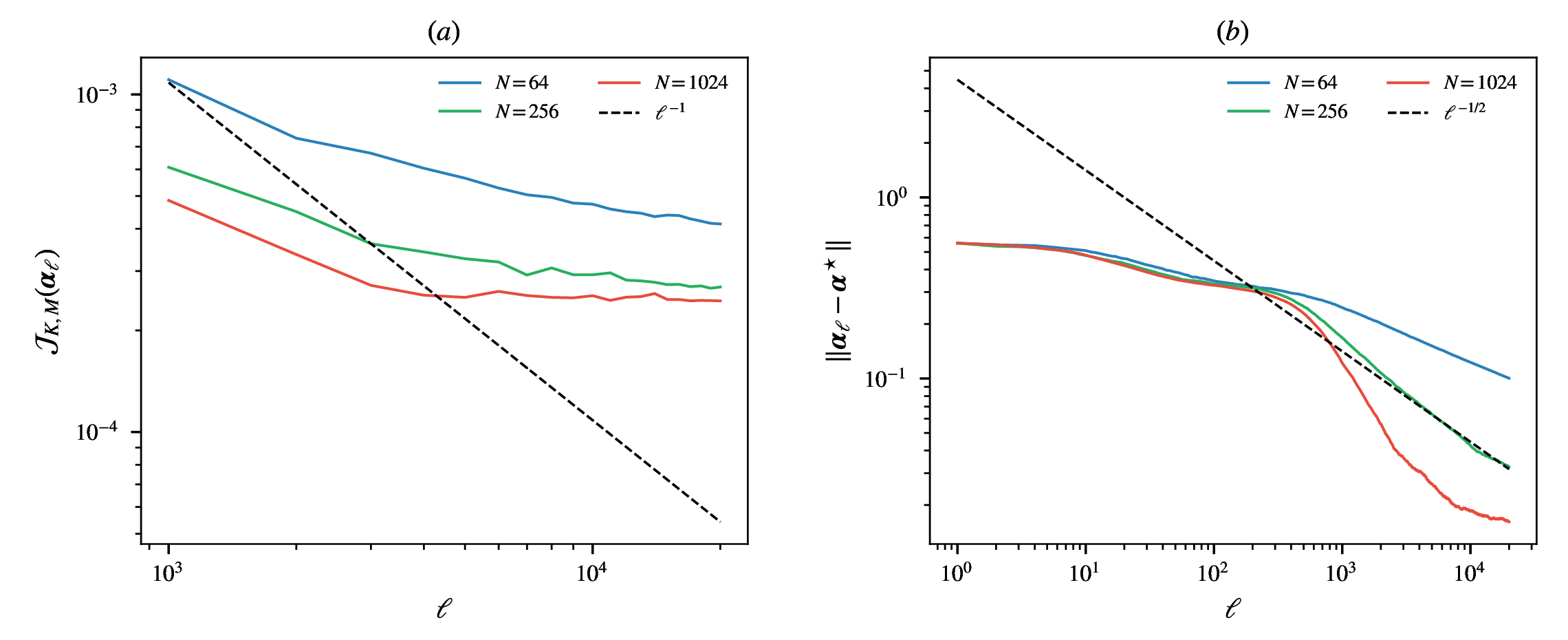}
    \vspace{-8pt}
    \caption{Recovery of $\boldsymbol{\alpha}^\star=(1,0.8,1)$ by
    Algorithm~\ref{alg:crossed-sgd} on the bistable-latent model, step
    $\eta_\ell=0.3/(\ell+50)$, $2\times10^4$ iterations, $N\in\{64,256,1024\}$, eight seeds.
    $(a)$~weak loss versus $\ell$.  $(b)$~parameter error versus $\ell$ against the
    $\ell^{-1/2}$ guide line.}
    \label{fig:bimod_convergence}
\end{figure}

\subsubsection{Neural-ODE black-box recovery}
\label{subsub:neural_bimodal}

The parametric recovery assumes the model form~\eqref{eq:bimodal_ode} is known.
To test whether the hidden mode is genuinely required, rather than an artifact of that
form, we discard the form and fit the observed marginals with two black-box vector
fields trained by Algorithm~\ref{alg:crossed-sgd}.
The first is a one-dimensional autonomous field $\dot{w}=\mathrm{NN}(w)$; the second
augments the state with a single latent coordinate initialized as
$z_0\sim\mathcal{N}(0,1)$ and learns the autonomous field
$(\dot{w},\dot{z})=\mathrm{NN}(w,z)$.
Both networks have two hidden layers of width $32$ with $\tanh$ activations, are
integrated by RK4, and are trained on the same test functions
$\mathbb{E}[w],\mathbb{E}[w^2],\mathbb{E}[w^4]$ at the same $M=5$ times; only the
marginal of $w$ enters the loss.
The crossed estimator~\eqref{eq:crossed-estimator} is realized here by a split
construction: one set of characteristics supplies the detached moment misfit and an independent set the
differentiated moments, so the preconditioned gradient remains unbiased.

Figure~\ref{fig:neural_bimodal} reports the outcome over four seeds for each field.
Both fields reduce the trained moments to a comparable weak loss of order $10^{-3}$, but the marginals they recover differ sharply.
We measure the discrepancy between a recovered marginal and the true marginal of $w$ by the
$1$-Wasserstein distance \cite{villani2009optimal}: for two probability measures $\mu$ and $\nu$ on
$\mathbb{R}$,
\begin{equation}
\label{eq:wasserstein}
W_1(\mu,\nu)
:=
\inf_{\gamma\in\Gamma(\mu,\nu)}\int_{\mathbb{R}^2}|x-y|\,d\gamma(x,y)
=
\int_{\mathbb{R}}\bigl|F_\mu(x)-F_\nu(x)\bigr|\,dx,
\end{equation}
where $\Gamma(\mu,\nu)$ is the set of couplings with marginals $\mu$ and $\nu$, and $F_\mu,F_\nu$ are their cumulative distribution functions; the second equality is the closed form valid on $\mathbb{R}$.
It is the minimal cost of transporting one distribution onto the other under the Euclidean ground cost, and it metrizes convergence in distribution, so a small $W_1$ certifies that
the whole marginal shape, not merely the trained moments, is matched.
The one-dimensional field is less accurate and worsens with time: its $W_1$ distance to the true
marginal grows from $0.12$ to $0.28$ across the observation window, about three times that
of the latent field, whose $W_1$ stays near $0.08$, and its density at the final time
misplaces the two peaks.
Matching a few moments does not pin down the flow; a scalar autonomous field can meet
the moment targets while reconstructing the wrong marginal sequence, whereas adding one latent
dimension improves reconstruction of the sequence the hidden mode generates.

\begin{figure}[h!]
    \centering
    \includegraphics[width=\linewidth]{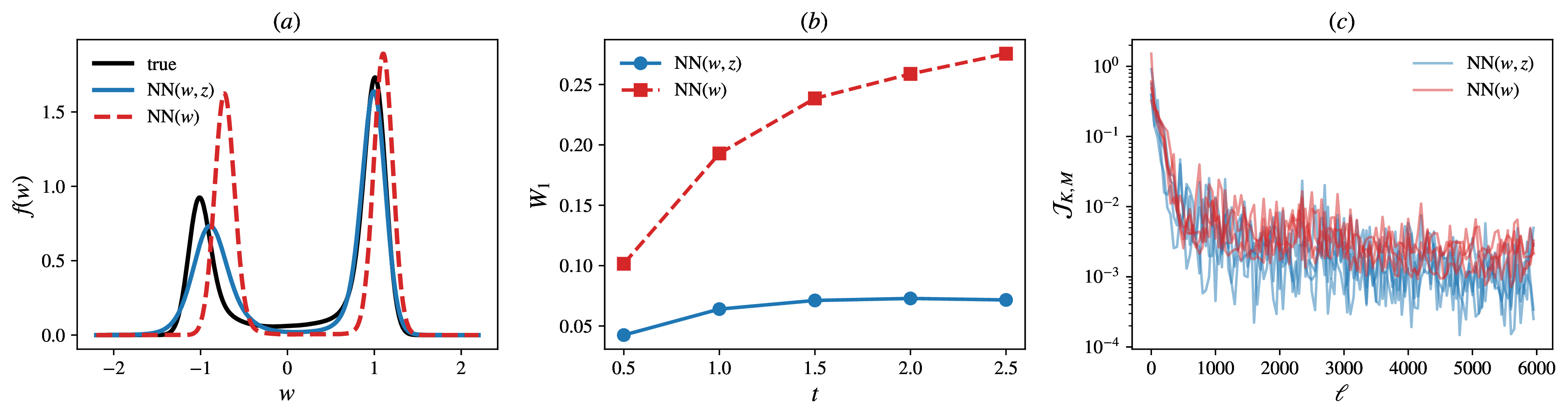}
    \vspace{-8pt}
    \caption{Black-box recovery of the observed marginal of $w$ by a one-dimensional
    autonomous field $\dot{w}=\mathrm{NN}(w)$ and a latent field
    $(\dot{w},\dot{z})=\mathrm{NN}(w,z)$, both trained by Algorithm~\ref{alg:crossed-sgd}
    on $\mathbb{E}[w],\mathbb{E}[w^2],\mathbb{E}[w^4]$ at $M=5$ times, four seeds each.
    $(a)$~learned marginal of $w$ at $t=2.5$ against the true marginal.
    $(b)$~Wasserstein distance of the learned marginal to the true marginal versus time,
    for a single representative seed of each field.
    $(c)$~weak loss on the trained moments versus iteration.}
    \label{fig:neural_bimodal}
\end{figure}

\subsection{Experiment 4: Particle-in-flow drag law inference}
\label{subsec:particle_drag}

Consider a particle suspended in a prescribed cellular flow, subject to a drag force with unknown constitutive law.
This experiment tests the method on a five-dimensional state space with degenerate observability: only the particle's position is directly observed, yet the drag parameters govern the velocity dynamics.

\subsubsection{Model description}

A point particle at position $\boldsymbol{x}_p \in \mathbb{R}^2$ with velocity $\boldsymbol{u}_p \in \mathbb{R}^2$ moves in a steady incompressible flow,
\begin{equation}
    \boldsymbol{u}(\boldsymbol{x}) = (\sin 2\pi x_1 \cos 2\pi x_2,\, -\cos 2\pi x_1 \sin 2\pi x_2)^\top
\end{equation}
The drag coupling introduces a scalar efficiency factor $\chi$, giving a five-dimensional state $\boldsymbol{q} = (\boldsymbol{x}_p, \boldsymbol{u}_p, \chi) \in \mathbb{R}^5$:
\begin{equation}
\label{eq:particle_ode}
\dot{\boldsymbol{x}}_p = \boldsymbol{u}_p, \qquad
\dot{\boldsymbol{u}}_p = \chi\, F(|\boldsymbol{u} - \boldsymbol{u}_p|)\, \frac{\boldsymbol{u} - \boldsymbol{u}_p}{|\boldsymbol{u} - \boldsymbol{u}_p|}, \qquad
\dot{\chi} = 0,
\end{equation}
where $F(s) = a\,s + b\,s^2$ is the Stokes--Oseen drag law with unknown parameters $\boldsymbol{\alpha} = (a, b)$.
In the notation of the general formulation~\eqref{eqn:ODE}, the state is $\boldsymbol{q} = (\boldsymbol{x}_p, \boldsymbol{u}_p, \chi)$ and the vector field, autonomous because the background flow $\boldsymbol{u}$ is steady, is
\[
\boldsymbol{R}(\boldsymbol{q};\boldsymbol{\alpha}) = \Bigl(\boldsymbol{u}_p,\;\; \chi\, F(|\boldsymbol{u} - \boldsymbol{u}_p|)\, \tfrac{\boldsymbol{u} - \boldsymbol{u}_p}{|\boldsymbol{u} - \boldsymbol{u}_p|},\;\; 0\Bigr),
\]
so the drag parameters enter $\boldsymbol{R}$ only through the velocity block.
The initial distribution is $\boldsymbol{x}_p \sim \mathcal{N}(0, 0.2^2 I_2)$, $\boldsymbol{u}_p \sim \mathcal{N}(0, 0.1^2 I_2)$, $\chi \sim \mathcal{N}(1, 0.1^2)$.
Target moments are computed via five-dimensional Gauss--Legendre quadrature ($20^2 \times 10^2 \times 10 = 400\,000$ points) over the full initial distribution, at $M = 5$ observation times in $(0, 1]$.

\subsubsection{Observability and test function design}

The drag force acts on velocity, not position, but only position is observed.
We therefore use $K = 6$ position-only test functions,
\[
\{\,x_{p,1}^2,\;\; x_{p,2}^2,\;\; \cos(2\pi x_{p,1}),\;\; \cos(2\pi x_{p,2}),\;\; \cos(2\pi x_{p,1})\cos(2\pi x_{p,2}),\;\; x_{p,1}^4 + x_{p,2}^4\,\}.
\]
The second moments $\mathbb{E}[x_{p,j}^2]$ and the fourth moment $\mathbb{E}[x_{p,1}^4 + x_{p,2}^4]$ measure how far particles diffuse from the origin, which depends on how quickly the drag synchronizes $\boldsymbol{u}_p$ with $\boldsymbol{u}(\boldsymbol{x})$; the period-1 Fourier modes $\mathbb{E}[\cos(2\pi x_{p,j})]$ and the checkerboard $\mathbb{E}[\cos(2\pi x_{p,1})\cos(2\pi x_{p,2})]$ are sensitive to how particles localize within the cellular flow's vortex cells.
Together these probe the spatial structure that the drag parameters $(a, b)$ imprint on the position marginal.

\subsubsection{Results}

Starting from the initial guess $(a_0, b_0) = (0.5, 0.2)$ with true parameters $(a, b) = (1.0, 0.5)$, Algorithm~\ref{alg:crossed-sgd} with $N = 512$ particles per set and the Shampoo preconditioner reaches the truth's neighborhood within $\sim 50$ iterations.  Across $10^3$ iterations the trailing-mean estimate over the last $200$ steps is $(\hat a, \hat b) = (1.054, 0.493)$, biased by $5.4\%$ and $1.4\%$ respectively.
Figure~\ref{fig:particle_conv} shows the loss and parameter trajectories.

Figure~\ref{fig:particle_marginals} compares the position marginal at $t = T_{\max} = 1$ for the true, initial-guess, and recovered parameters.
Figure~\ref{fig:particle_moments} shows the six position-only moments as functions of time.

\begin{figure}[ht]
    \centering
    \includegraphics[width=0.85\linewidth]{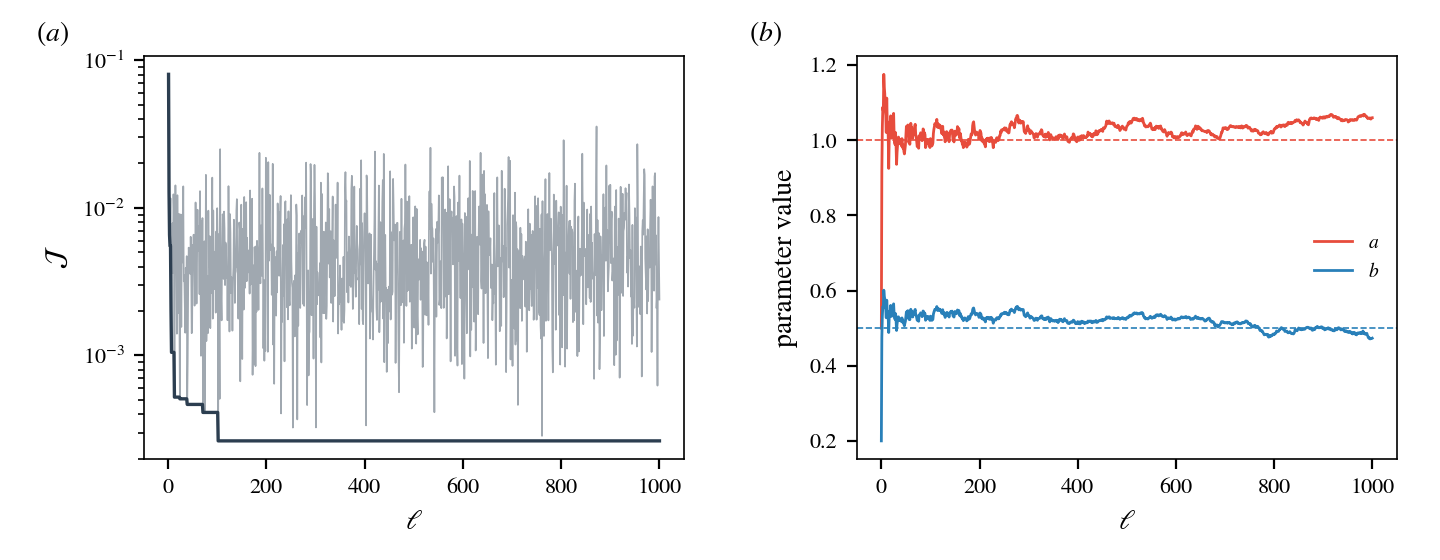}
    \caption{Particle drag-law inference with $F(s) = as + bs^2$, $N = 512$ particles per set.
    $(a)$~Loss versus iteration (gray: raw iterate; solid: running minimum).
    $(b)$~Parameter trajectories with dashed lines at the true values $(a, b) = (1, 0.5)$.}
    \label{fig:particle_conv}
\end{figure}

\begin{figure}[ht]
    \centering
    \includegraphics[width=\linewidth, trim={2 40 100 40}, clip]{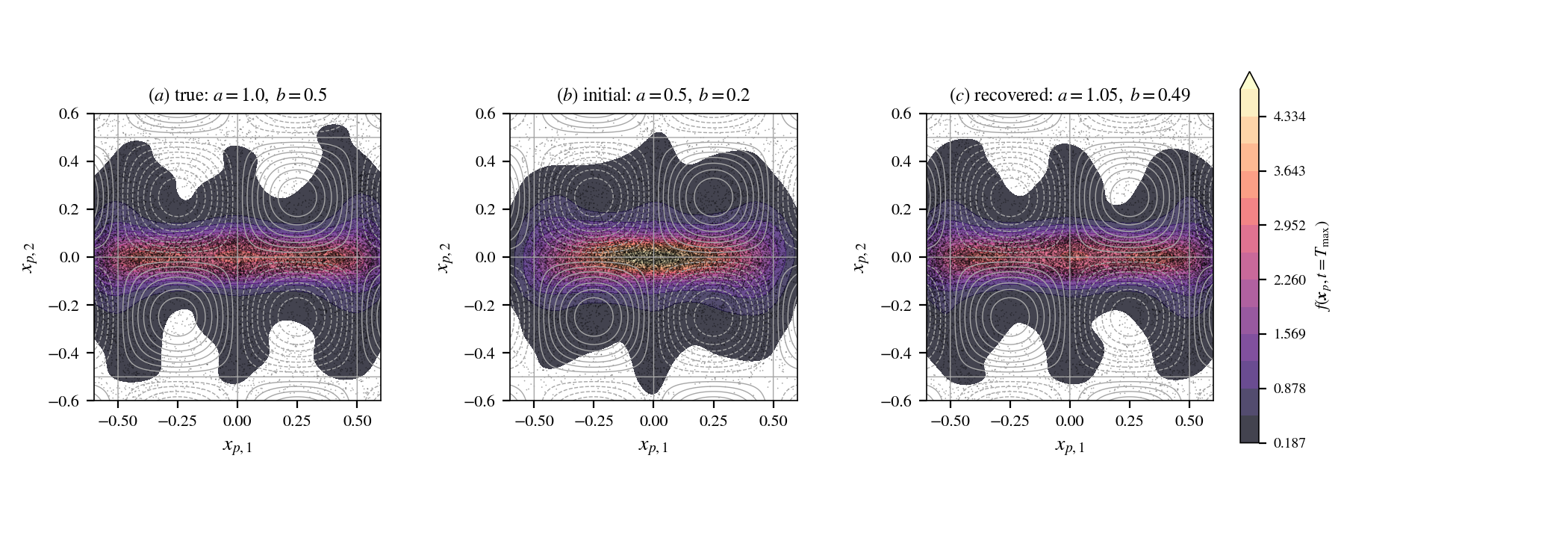}
    \vspace{-20pt}
    \caption{Particle position marginal at $t = 1$ under $(a)$~the true parameters, $(b)$~the initial guess, and $(c)$~the recovered parameters.}
    \label{fig:particle_marginals}
\end{figure}

\begin{figure}[ht]
    \centering
    \includegraphics[width=\linewidth]{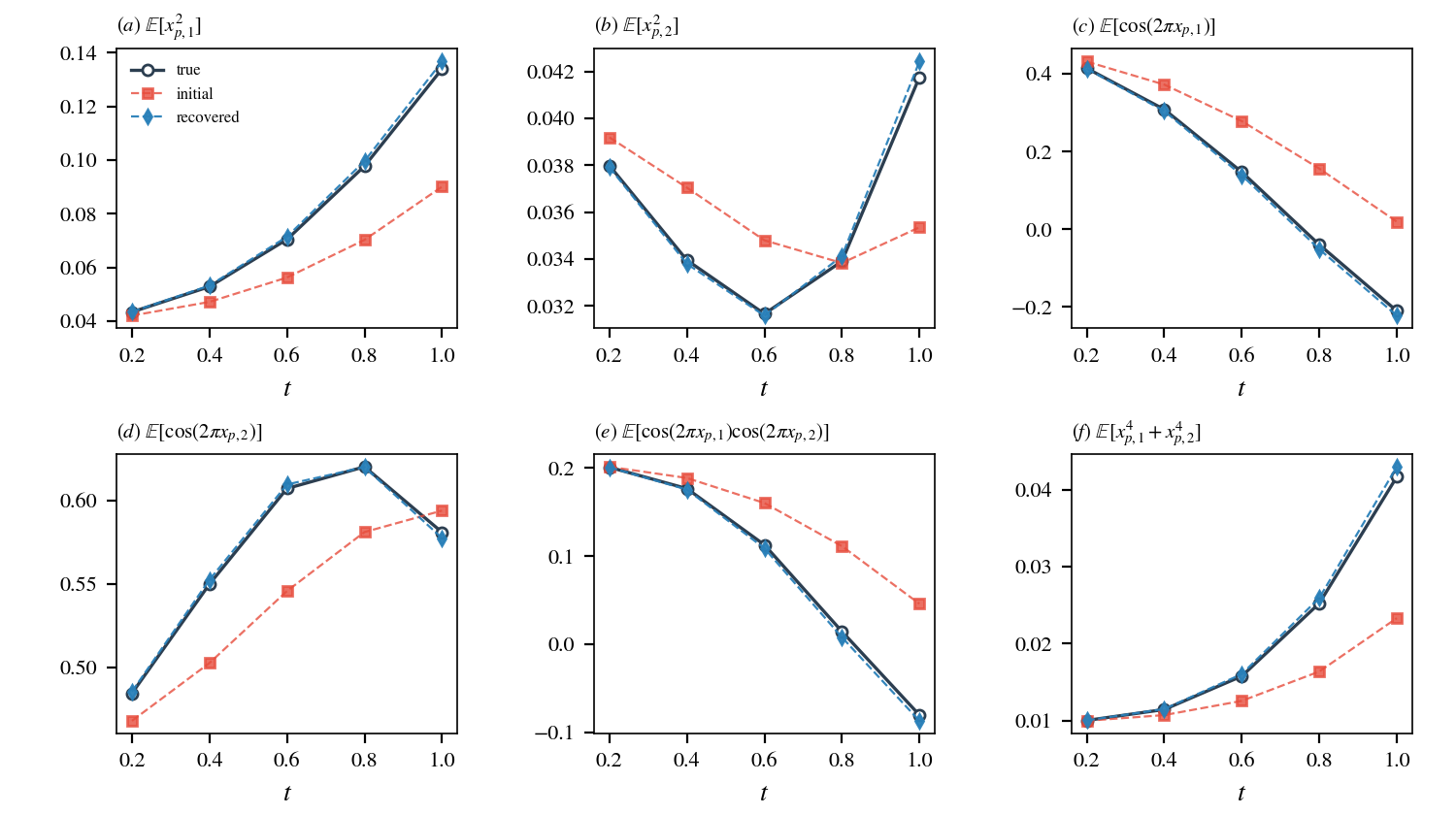}
    \caption{The six position-only observation moments versus time under the true parameters (black circles), the initial guess (red squares), and the recovered parameters (blue diamonds).}
    \label{fig:particle_moments}
\end{figure}

\section{Conclusion}
\label{sec:conclusion}

This paper has presented a stochastic gradient framework for parameter inference in dynamical systems from marginal distribution data.
By exploiting the characteristic representation of the Liouville equation, the method avoids discretizing the density PDE in full phase space and instead constructs unbiased gradient estimators via pathwise sensitivities
and a crossed U-statistic construction.

The theoretical analysis establishes that the gradient variance scales as $O(1/N)$ in the number of sampled characteristics and grows exponentially with the time
horizon at a rate governed by the Lipschitz constant of the vector field. The conditioning of the sensitivity Gramian, degraded both by chaotic amplification of the tangent dynamics and by marginalization over the latent coordinates, motivates the use of full-matrix preconditioning.

Four numerical experiments validated the framework.
The linear hidden-mode experiment ($d = 3$, three parameters, single observable) validated both halves of the asymptotic rate on a well-conditioned, uniquely identifiable system whose sensitivities, sensitivity Gramian, and identifiability are available in closed form.
The nonlinear Gompertz experiment carried the same convergence check to a genuinely nonlinear vector field for which the moments remain exactly known, recovering the true parameters at the predicted rate.
The bistable-latent experiment showed that a hidden mode can turn a unimodal initial marginal into a bimodal observed one, for which the tested scalar black-box model was substantially less accurate than the latent model,
and recovered the physical parameters, including the latent decay, from the observed marginal alone.
The particle-in-flow experiment ($d = 5$, Stokes--Oseen drag law) showed that position-only test functions can identify drag parameters acting through the unobserved velocity dynamics.

The main limitation is the exponential variance growth with time horizon~$T$, which restricts the method to moderate horizons or requires windowing strategies.
Natural extensions include variance reduction via control variates, adaptation to stochastic dynamics governed by the Fokker--Planck equation, and the development of adaptive or exponential integrators for stiff velocity fields.

\section*{Acknowledgments}
This work was supported by the Air Force Office of Scientific Research under award FA9550-23-1-0405 and by the National Science Foundation under awards NSF-2431610.
The authors used ChatGPT to assist with drafting and editing portions of the manuscript, with the derivation of a improved bound for the convergence in Section 4, the tikz schematic in Figure 2, and the problem setup of Experiment 3. All definitions, theorem statements, proofs, and numerical results were verified by the authors, who take full responsibility for the content of this paper.

\appendix

\bibliographystyle{siamplain}
\bibliography{references}

@article{natarajan2021high,
  title={A high-order semi-Lagrangian method for the consistent Monte-Carlo solution of stochastic Lagrangian drift--diffusion models coupled with Eulerian discontinuous spectral element method},
  author={Natarajan, H and Popov, PP and Jacobs, GB},
  journal={Computer Methods in Applied Mechanics and Engineering},
  volume={384},
  pages={114001},
  year={2021},
  publisher={Elsevier}
}

@article{dominguez2024high,
  title={High-order Lagrangian algorithms for Liouville models of particle-laden flows},
  author={Dom{\'\i}nguez-V{\'a}zquez, Daniel and Castiblanco-Ballesteros, Sergio A and Jacobs, Gustaaf B and Tartakovsky, Daniel M},
  journal={Journal of Computational Physics},
  volume={515},
  pages={113281},
  year={2024},
  publisher={Academic Press}
}

@Article{	  robbins1951stochastic,
  title		= {A stochastic approximation method},
  author	= {Robbins, Herbert and Monro, Sutton},
  journal	= {The annals of mathematical statistics},
  pages		= {400--407},
  year		= {1951},
  publisher	= {JSTOR}
}

@Book{		  kushner2003stochastic,
  title		= {Stochastic approximation and recursive algorithms and
		  applications},
  author	= {Kushner, Harold J and Yin, G George},
  year		= {2003},
  publisher	= {Springer}
}

@Article{	  lions1971optimal,
  title		= {Optimal control of systems governed by partial
		  differential equations},
  author	= {Lions, Jacques-Louis},
  journal	= {Springer},
  year		= {1971}
}

@Book{		  evensen2009data,
  title		= {Data Assimilation: The Ensemble Kalman Filter},
  author	= {Evensen, Geir},
  publisher	= {Springer},
  year		= {2009}
}

@Book{		  risken1996fokkerplanck,
  title		= {The Fokker--Planck Equation: Methods of Solution and
		  Applications},
  author	= {Risken, Hannes},
  publisher	= {Springer},
  year		= {1996}
}

@Article{	  sun2019moment,
  title		= {Moment closure approximations in stochastic dynamics},
  author	= {Sun, Jie and Bollt, Erik},
  journal	= {Physica D},
  volume	= {390},
  pages		= {1--19},
  year		= {2019}
}

@article{rubanova2019latent,
  title={Latent ordinary differential equations for irregularly-sampled time series},
  author={Rubanova, Yulia and Chen, Ricky TQ and Duvenaud, David K},
  journal={Advances in neural information processing systems},
  volume={32},
  year={2019}
}

@article{chen2018neural,
  title={Neural ordinary differential equations},
  author={Chen, Ricky TQ and Rubanova, Yulia and Bettencourt, Jesse and Duvenaud, David K},
  journal={Advances in neural information processing systems},
  volume={31},
  year={2018}
}

@article{anderson1982reverse,
  title={Reverse-time diffusion equation models},
  author={Anderson, Brian DO},
  journal={Stochastic Processes and their Applications},
  volume={12},
  number={3},
  pages={313--326},
  year={1982},
  publisher={Elsevier}
}

@article{kobyzev2020normalizing,
  title={Normalizing flows: An introduction and review of current methods},
  author={Kobyzev, Ivan and Prince, Simon JD and Brubaker, Marcus A},
  journal={IEEE transactions on pattern analysis and machine intelligence},
  volume={43},
  number={11},
  pages={3964--3979},
  year={2020},
  publisher={IEEE}
}

@article{papamakarios2021normalizing,
  title={Normalizing flows for probabilistic modeling and inference},
  author={Papamakarios, George and Nalisnick, Eric and Rezende, Danilo Jimenez and Mohamed, Shakir and Lakshminarayanan, Balaji},
  journal={Journal of Machine Learning Research},
  volume={22},
  number={57},
  pages={1--64},
  year={2021}
}

@inproceedings{rombach2022high,
  title={High-resolution image synthesis with latent diffusion models},
  author={Rombach, Robin and Blattmann, Andreas and Lorenz, Dominik and Esser, Patrick and Ommer, Bj{\"o}rn},
  booktitle={Proceedings of the IEEE/CVF conference on computer vision and pattern recognition},
  pages={10684--10695},
  year={2022}
}

@Article{	  song2021score,
  title		= {Score-based generative modeling through stochastic
		  differential equations},
  author	= {Song, Yang and Sohl-Dickstein, Jascha and Kingma, Diederik
		  P. and Kumar, Abhishek and Ermon, Stefano and Poole, Ben},
  journal	= {International Conference on Learning Representations},
  year		= {2021}
}

@Article{	  song2020sde,
  title		= {Generative modeling by estimating gradients of the data
		  distribution},
  author	= {Song, Yang and Ermon, Stefano},
  journal	= {Advances in Neural Information Processing Systems},
  volume	= {33},
  year		= {2020}
}

@Article{	  leith1974theoretical,
  title		= {Theoretical skill of Monte Carlo forecasts},
  author	= {Leith, C. E.},
  journal	= {Monthly Weather Review},
  volume	= {102},
  pages		= {409--418},
  year		= {1974}
}

@Book{		  nagy1970harmonic,
  title		= {Harmonic Analysis of Operators on {Hilbert} Space},
  author	= {Sz.-Nagy, B\'ela and Foia\c{s}, Ciprian},
  publisher	= {North-Holland},
  year		= {1970}
}

@Book{		  breuer2002open,
  title		= {The Theory of Open Quantum Systems},
  author	= {Breuer, Heinz-Peter and Petruccione, Francesco},
  publisher	= {Oxford University Press},
  year		= {2002}
}

@Article{	  mori1965transport,
  title		= {Transport, Collective Motion, and {Brownian} Motion},
  author	= {Mori, Hazime},
  journal	= {Progress of Theoretical Physics},
  volume	= {33},
  number	= {3},
  pages		= {423--455},
  year		= {1965}
}

@Article{	  zwanzig1961memory,
  title		= {Memory Effects in Irreversible Thermodynamics},
  author	= {Zwanzig, Robert},
  journal	= {Physical Review},
  volume	= {124},
  number	= {4},
  pages		= {983--992},
  year		= {1961}
}

@Article{	  rabiner1989hmm,
  title		= {A tutorial on hidden {Markov} models and selected
		  applications in speech recognition},
  author	= {Rabiner, Lawrence R.},
  journal	= {Proceedings of the IEEE},
  volume	= {77},
  number	= {2},
  pages		= {257--286},
  year		= {1989}
}

@book{villani2009optimal,
  title={Optimal transport: old and new},
  author={Villani, C{\'e}dric and others},
  volume={338},
  year={2009},
  publisher={Springer}
}

@article{lee2019unbiased,
  title={Unbiased approximations of products of expectations},
  author={Lee, Anthony and Tiberi, Simone and Zanella, Giacomo},
  journal={Biometrika},
  volume={106},
  number={3},
  pages={708--715},
  year={2019},
  publisher={Oxford University Press}
}

@Article{	  bottou2018optimization,
  title		= {Optimization methods for large-scale machine learning},
  author	= {Bottou, L\'eon and Curtis, Frank E. and Nocedal, Jorge},
  journal	= {SIAM Review},
  volume	= {60},
  number	= {2},
  pages		= {223--311},
  year		= {2018}
}

@book{borkar2008stochastic,
  title={Stochastic approximation: a dynamical systems viewpoint},
  author={Borkar, Vivek S and Borkar, Vivek S},
  volume={100},
  year={2008},
  publisher={Springer}
}

@article{fabian1968asymptotic,
  title={On asymptotic normality in stochastic approximation},
  author={Fabian, Vaclav},
  journal={The Annals of Mathematical Statistics},
  pages={1327--1332},
  year={1968},
  publisher={JSTOR}
}

@book{murray2013existence,
  title={Existence theorems for ordinary differential equations},
  author={Murray, Francis J and Miller, Kenneth S},
  year={2013},
  publisher={Courier Corporation}
}

@Article{	  gupta2018shampoo,
  title		= {Shampoo: Preconditioned Stochastic Tensor Optimization},
  author	= {Gupta, Vineet and Koren, Tomer and Singer, Yoram},
  journal	= {International Conference on Machine Learning},
  year		= {2018}
}

@Article{	  kingma2014adam,
  title		= {Adam: A Method for Stochastic Optimization},
  author	= {Kingma, Diederik P. and Ba, Jimmy},
  journal	= {International Conference on Learning Representations},
  year		= {2015}
}

@article{10.1063/5.0207403,
    author = {Domínguez-Vázquez, Daniel and Jacobs, Gustaaf B. and Tartakovsky, Daniel M.},
    title = {Liouville models of particle-laden flow},
    journal = {Physics of Fluids},
    volume = {36},
    number = {6},
    pages = {063303},
    year = {2024},
    month = {06},
    issn = {1070-6631},
    doi = {10.1063/5.0207403},
    url = {https://doi.org/10.1063/5.0207403},
    eprint = {https://pubs.aip.org/aip/pof/article-pdf/doi/10.1063/5.0207403/19974821/063303_1_5.0207403.pdf},
}

@article{pavliotis2014stochastic,
  title={Stochastic processes and applications},
  author={Pavliotis, Grigorios A},
  journal={Texts in applied mathematics},
  volume={60},
  pages={41--43},
  year={2014},
  publisher={Springer}
}

@ARTICLE{nodozi2024neural,
  author={Nodozi, Iman and Yan, Charlie and Khare, Mira and Halder, Abhishek and Mesbah, Ali},
  journal={IEEE Transactions on Control Systems Technology}, 
  title={Neural Schrödinger Bridge With Sinkhorn Losses: Application to Data-Driven Minimum Effort Control of Colloidal Self-Assembly}, 
  year={2024},
  volume={32},
  number={3},
  pages={960-973},
  doi={10.1109/TCST.2023.3337588}}

@article{kgb6-k3zm,
  title = {Localization of sources in weakly nonlinear fluid systems using linear and quadratic sensitivity analysis},
  author = {Wang, Qi and You, Zejian},
  journal = {Phys. Rev. Fluids},
  pages = {},
  year = {2026},
  month = {Mar},
  publisher = {American Physical Society},
  doi = {10.1103/kgb6-k3zm},
  url = {https://link.aps.org/doi/10.1103/kgb6-k3zm}
}

@article{farea2024understanding,
  title={Understanding physics-informed neural networks: Techniques, applications, trends, and challenges},
  author={Farea, Amer and Yli-Harja, Olli and Emmert-Streib, Frank},
  journal={Ai},
  volume={5},
  number={3},
  pages={1534--1557},
  year={2024},
  publisher={MDPI}
}

@article{benamou2000computational,
  title={A computational fluid mechanics solution to the Monge-Kantorovich mass transfer problem},
  author={Benamou, Jean-David and Brenier, Yann},
  journal={Numerische Mathematik},
  volume={84},
  number={3},
  pages={375--393},
  year={2000},
  publisher={Springer-Verlag Berlin/Heidelberg}
}

@article{Wang_Wang_Zaki_2022, 
title={What is observable from wall data in turbulent channel flow?}, 
volume={941}, 
DOI={10.1017/jfm.2022.295}, 
journal={Journal of Fluid Mechanics}, 
author={Wang, Qi and Wang, Mengze and Zaki, Tamer A.}, 
year={2022}, 
pages={A48}}

@article{Wang_Zaki_2025, 
title={Domain of dependence for wall-pressure measurements in high-speed boundary layers}, 
volume={1009}, DOI={10.1017/jfm.2025.224}, 
journal={Journal of Fluid Mechanics}, 
author={Wang, Qi and Zaki, Tamer A.}, 
year={2025}, 
pages={A67}}

@article{dominguez2022inference,
    author = {Domínguez-Vázquez, Daniel and Escobar-Castaneda, Nicolas and Wang, Qi and Jacobs, Gustaaf B.},
    title = {Inference of inertial particle dynamics from limited measurements},
    journal = {Physics of Fluids},
    volume = {37},
    number = {4},
    pages = {043361},
    year = {2025},
    month = {04},
    issn = {1070-6631},
    doi = {10.1063/5.0257997},
}

@InProceedings{lipman2023flow,
  title     = {Flow Matching for Generative Modeling},
  author    = {Lipman, Yaron and Chen, Ricky T. Q. and Ben-Hamu, Heli and Nickel, Maximilian and Le, Matt},
  booktitle = {International Conference on Learning Representations (ICLR)},
  year      = {2023},
}

@article{daniel2024liouville,
    author = {Domínguez-Vázquez, Daniel and Jacobs, Gustaaf B. and Tartakovsky, Daniel M.},
    title = {Liouville models of particle-laden flow},
    journal = {Physics of Fluids},
    volume = {36},
    number = {6},
    pages = {063303},
    year = {2024},
    month = {06},
    issn = {1070-6631},
    doi = {10.1063/5.0207403},
}

@InProceedings{liu2016stein,
  title     = {Stein Variational Gradient Descent: A General Purpose {B}ayesian Inference Algorithm},
  author    = {Liu, Qiang and Wang, Dilin},
  booktitle = {Advances in Neural Information Processing Systems (NeurIPS)},
  year      = {2016},
}

@article{cranmer2020frontier,
  title={The frontier of simulation-based inference},
  author={Cranmer, Kyle and Brehmer, Johann and Louppe, Gilles},
  journal={Proceedings of the National Academy of Sciences},
  volume={117},
  number={48},
  pages={30055--30062},
  year={2020},
  publisher={National Academy of Sciences}
}

@book{birdsall,
  title={Plasma Physics via Computer Simulation},
  author={Birdsall, Charles K and Langdon, A Bruce},
  year={2004},
  publisher={CRC Press}
}

@article{hoeffding1948ustatistic,
  title={A class of statistics with asymptotically normal distribution},
  author={Hoeffding, Wassily},
  journal={The Annals of Mathematical Statistics},
  volume={19},
  number={3},
  pages={293--325},
  year={1948},
  publisher={Institute of Mathematical Statistics}
}
\end{document}